\documentclass[a4paper,12pt,reqno]{article}
\usepackage[english]{babel}
\usepackage{amsmath,amsfonts,amssymb,amsthm,bbm}
\usepackage{graphics,epsfig,psfrag} %%add this and next lines if pictures should be in esp format
\usepackage{paralist,subfigure,multirow,float}
\usepackage{hyperref} 
\usepackage[active]{srcltx} %%allows you to jump from your editor directly into the associated position of the DVI file
\usepackage[dvipsnames]{xcolor}
\usepackage{soul} %%allows you to use the \st{...} command, which produces a raised line over the text. If \setstcolor{red} is used, the line is red
\usepackage[latin1]{inputenc}
\usepackage{mathrsfs}
\usepackage[OT1]{fontenc}

\newtheorem{theorem}{Theorem}[section]
\newtheorem{corollary}[theorem]{Corollary}
\newtheorem{lemma}[theorem]{Lemma}
\newtheorem{proposition}[theorem]{Proposition}
\newtheorem{remark}[theorem]{Remark}

\newtheorem{hypothesis}[theorem]{Hypothesis}
\newtheorem{definition}[theorem]{Definition}

\theoremstyle{definition}

 \usepackage{autonum} %% displays only cited formulas;

\DeclareMathOperator{\dv}{div}

\def\N{\mathbb{N}}
\def\Z{\mathbb{Z}}

\def\R{\mathbb{R}}

\let\e=\varepsilon
\def\epsilon{\varepsilon}

\let\t=\tilde
\let\ol=\overline
\let\ul=\underline
\let\.=\cdot
\let\0=\emptyset

\let\mc=\mathcal

\def\1{\mathbbm{1}}
\def\Sph{{\mathbb{S}}^{N-1}}
\newcommand{\su}[2]{\genfrac{}{}{0pt}{}{#1}{#2}}

\def\hat{\widehat}
\def\tilde{\widetilde}

\def\seq#1{(#1_n)_{n\in\N}}

\newenvironment{formula}[1]{\begin{equation}\label{#1}}{\end{equation}\noindent}

\numberwithin{equation}{section}% equation numbering based on section
\newcommand{\be}{\begin{equation}}
\newcommand{\ee}{\end{equation}}
\newcommand{\baa}{\begin{array}}
\newcommand{\eaa}{\end{array}}
\newcommand{\ba}{\begin{eqnarray}}
\newcommand{\ea}{\end{eqnarray}}
\def\Fi#1{\begin{formula}{#1}}
\def\Ff{\end{formula}\noindent}
\def\BS{\color{Bittersweet}}

\def\PTF{pulsating traveling front}

\def\W{\mc{W}}

\begin{document}
\date{}
\title{\bf{Reaction-diffusion equations in periodic media: spreading speeds and spreading sets\thanks{This work has received funding from Excellence Initiative of Aix-Marseille Universit\'e~-~A*MIDEX, a French ``Investissements d'Avenir'' programme, from the French ANR ReaCh (ANR-23-CE40-0023) project, and from the European Union -- Next Generation EU, PRIN project 2022W58BJ5 ``PDEs and optimal control methods in mean field games, population dynamics and multi-agent models''. The second author is grateful to the hospitality of Universit\`a degli Studi di Roma La Sapienza, where part of this work was done. The first author is supported by the fundamental research funds for the central universities and the National Natural Science Foundation of China (No. 12471201).}}}
\author{Hongjun Guo$^{\hbox{\small{ a}}}$, Fran\c cois Hamel$^{\hbox{\small{ b}}}$ and Luca Rossi$^{\hbox{\small{ c}}}$\\
\\
\footnotesize{$^{\hbox{a }}$School of Mathematical Sciences, Key Laboratory of Intelligent Computing and Applications }\\
\footnotesize{(Ministry of Education), Institute for Advanced Study, Tongji University, Shanghai, China}\\
\footnotesize{$^{\hbox{b }}$Aix Marseille Univ, CNRS, I2M, Marseille, France}\\
\footnotesize{$^{\hbox{c }}$SAPIENZA Univ Roma, Istituto ``G.~Castelnuovo'', Roma, Italy}}
\maketitle
	
\begin{abstract}
\noindent{We consider solutions of space-periodic reaction-diffusion-advection equations in the whole space $\R^N$, with general unbounded initial support. We show the existence of spreading sets for the solutions at large time, and variational formulas for the spreading speeds. The location of the upper-level sets of a solution with unbounded initial support is estimated in terms of that initial support and the spreading set for the solutions with bounded initial supports. The results hold for the standard classes of Fisher-KPP, ignition, or bistable reactions. The general results apply to the cases of $V$-shaped and $\Lambda$-shaped initial supports. The proofs rely in particular on upper estimates of what we call retracting solutions.}
\end{abstract}
	
\tableofcontents

%-------------------------------------------------------------------------------
%-------------------------------------------------------------------------------

\section{Introduction and main results}\label{intro}

In this paper, we investigate the large-time dynamics of solutions of the following reaction-diffusion equation in periodic media
\begin{eqnarray}\label{main}
\partial_t u=\dv(A(x)\nabla u) +q(x)\cdot\nabla u + f(x,u),  \quad t>0,\ x\in\R^N,
\end{eqnarray}
where $\nabla$ denotes the gradient with respect to the $x$ variable, 
``$\ \cdot\ $" is the Euclidean scalar product in $\R^N$, and $N\ge1$ is an integer. 
By periodicity, we mean that $A(x)$, $q(x)$ and $f(x,\cdot)$ are $\Z^N$-periodic in $x$, that is, 
$$\forall\,h\in\Z^N,\qquad A(\.+h)\equiv A,\qquad q(\.+h)\equiv q,\qquad f(\.+h,\.)\equiv f.$$
The matrix field $x\mapsto A(x)=(a_{ij}(x))_{1\le i,j \le N}$ is symmetric positive definite, of class~$C^{1,\alpha}(\R^N)$ for some $\alpha\in (0,1)$ and the vector field $x\mapsto q(x)=(q_i(x))_{1\le i\le N}$ is of class~$C^{1}(\R^N)$ and satisfies
\be\label{hypq}
\hbox{div}\, q=0 \hbox{ in $\R^N$},\qquad \int_{(0,1)^N}\!q(x)dx=0.
\ee
The function $(x,s)\mapsto f(x,s)$ is of class $C^{0,\alpha}(\R^N\times [0,1])$, 
and~$\partial_s f=\frac{\partial f}{\partial s}$ exists and is of class $C^{0,\alpha}(\R^N\times [0,1])$.
We further assume that
\be\label{f01}
f(x,0)=f(x,1)=0\ \hbox{ for all } x\in\R^N.
\ee
For mathematical convenience, we extend $f$ by $0$ in $\R^N\times (\R\setminus [0,1])$. The extended function is then globally Lipschitz continuous with respect to $s\in\R$ uniformly in $x\in\R^N$. All the above assumptions will always be understood throughout the paper.

The initial datum $u_0$ for~\eqref{main} is assumed to be in the form $u_0=\1_U$, where $\1_U$ stands for the indicator function of a measurable set $U$, i.e.
\begin{eqnarray}\label{initial}
u_0(x)=\left\{\begin{array}{ll}
1 &\hbox{if $x\in U$},\vspace{3pt}\\
0 &\hbox{if $x\in\R^N\setminus U$}.
\end{array}
\right.
\end{eqnarray}
The set $U$ is called support of the function $u_0$, with a slight abuse of notation. The initial condition for~\eqref{main} is understood in the sense that $u(t,\cdot)\rightarrow u_0$ as $t\rightarrow 0^+$ in~$L_{loc}^1(\R^N)$ and the solution $u$ is the unique bounded classical solution of~\eqref{main} with initial condition~$u_0$. 
It follows from the parabolic maximum principle that
$$0\le u(t,x)\le 1\ \hbox{ for all $t>0$ and $x\in\R^N$},$$
with strict inequalities if the Lebesgue measures of $U$ and $\R^N\setminus U$ are positive, by the strong maximum principle. Nevertheless, the parabolic estimates tell us that $u$ stays close to $1$ or $0$ in subregions of $U$ or $\R^N\setminus U$ which are far away from $\partial U$, at any fixed time~$t$.
We are interested in the dynamics of these regions as time $t$ varies. More precisely, given some assumptions which guarantee the advantage of the steady state~$1$ over the steady state~$0$, our goal is to understand how the state~$1$ invades the state~$0$. To address this question, a fundamental aspect is the notion of directional spreading speed, defined as follows.

\begin{definition}\label{Def-Speed}
For a solution $u$ of \eqref{main}-\eqref{initial} and for a given unit vector $\xi \in\mathbb{S}^{N-1}$ $($the unit Euclidean sphere of $\R^N$$)$, a quantity $\omega(\xi)\in [0,+\infty]$ is called spreading speed $($in the direction~$\xi$$)$ if
\ba\label{conv:ss}
\left\{\begin{array}{lll}
u(t,ct\xi)\rightarrow 1 &\hbox{as $t\rightarrow +\infty$, $\ $for every $0\le c<\omega(\xi)$},\vspace{3pt}\\
u(t,ct \xi)\rightarrow 0 &\hbox{as $t\rightarrow +\infty$, $\ $for every $c>\omega(\xi)$}.
\end{array}
\right.
\ea
\end{definition}

The spreading speed $\omega(\xi)$, if any, somehow describes the asymptotic speed of the invasion front of the state $0$ by the state $1$ in the direction $\xi$. However, it is not enough to illustrate the global leading invasion front. Namely, there is no reference in Definition~\ref{Def-Speed} to the uniformity of the limits~\eqref{conv:ss} with respect to the directions $\xi\in\mathbb{S}^{N-1}$. Therefore, the following more general notion of spreading set will be used.

\begin{definition}\label{def:W}
We say that a solution $u$ to~\eqref{main} admits a spreading set
 $\W\subset\R^N$ if~$\W$ coincides  with the interior of its closure and satisfies 
\be\label{subset}
\lim_{t\to+\infty}\,\Big(\min_{x\in C}u(t, tx)\Big)=1\, \text{ for every non-empty compact set }C\subset \mc{W},
\ee
and
\be\label{superset}
\lim_{t\to+\infty}\,\Big(\max_{x\in C}u(t,tx)\Big)=0\, \text{ for every non-empty compact set }C\subset \hbox{\rm int}(\R^N\!\setminus\!\mc{W}).
\ee
If only \eqref{subset} $($resp. \eqref{superset}$)$ holds, we say that $\mc{W}$ is a spreading subset $($resp. superset$)$.
\end{definition}

The definition of the spreading set, if any, guarantees its uniqueness. Indeed, if~$\W_1$ and~$\W_2$ are spreading sets and if $x\in\W_1$, then $x\not\in{\rm int}(\R^N\!\setminus\!\W_2)$ by~\eqref{subset}-\eqref{superset}, that is,~$x\in\overline{\W_2}$. This implies $\W_1\subset\overline{\W_2}$ and then $\W_1={\rm int}(\W_1)\subset{\rm int}(\overline{\W_2})=\W_2$, and vice versa $\W_2\subset\W_1$. Furthermore, if $u$ has a spreading set $\W=\{r\xi:\xi\!\in\!\mathbb{S}^{N-1},\, 0\!\le\!r<\!\omega'(\xi)\}$ for some continuous map $\omega': \xi\mapsto \omega'(\xi)\in (0,+\infty]$, then $\omega'(\xi)$ is the spreading speed in the direction $\xi$, for each $\xi\in \mathbb{S}^{N-1}$. 

From Definition~\ref{def:W}, the boundary $\partial\W$ of the spreading set $\W$ dilated by the time, can be seen as the global leading front of the invasion of $0$ by~$1$. Namely, if a solution~$u$ admits a spreading set $\W$, and if its initial datum is compactly supported, then for any $\lambda\in(0,1)$, it holds that
$$\frac1t\,\big\{x\in\R^N:u(t,x)>\lambda\big\}\mathop{\longrightarrow}_{t\to+\infty}\W$$
in the sense of the Hausdorff distance, see \cite[Proposition~2.7]{GHR1}. If the initial datum is not compactly supported, then in general the convergence only occurs locally in space, cf.~\cite{HR1}.
 
The proof of the existence of $\W$, the characterization of $\partial\W$, and its relationship with the initial support set $U$, are some of the main goals of the paper. Before stating the main results, we first need to list the main hypotheses, related to various types of reaction terms~$f$.

%%%%%%%%%%%%%%%%%%%%%%%%%%%%%%%%%%%%%%%%%%%%%%%%%%%

\subsection{Main hypotheses}

By tracing back to the well-known result of Aronson and Weinberger~\cite{AW}, the spreading speed for the homogeneous reaction-diffusion
\be\label{HRD}
\partial_tu=\Delta u +f(u),  \quad t>0,\, x\in\R^N,
\ee
with compactly supported initial datum $0\le u_0\le 1$ such that $\|u_0\|_{L^{\infty}(\R^N)}>0$ and the monostable assumption $f(0)=f(1)=0$, $f(u)>0$ in $(0,1)$, can be characterized by the minimal speed $c^*$ of planar (or one-dimensional) fronts. Here, the minimal speed $c^*>0$ means that \eqref{HRD} in $\R$ admits traveling fronts $\phi(x -ct)$ such that $\phi(+\infty)=0<\phi<1=\phi(-\infty)$ if and only if $c\ge c^*$. More precisely, provided  $u(t,x)\rightarrow 1$ as $t\rightarrow +\infty$ locally uniformly in $\R^N$, which automatically holds if $f'(0)>0$ and more generally if $\liminf_{s\rightarrow 0^+} f(s)/s^{1+\frac{2}{N}}>0$~\cite{AW}, one has
\begin{eqnarray*}
\left\{\begin{array}{ll}
\displaystyle\lim_{t\rightarrow +\infty}\Big(\inf_{x\in B_{ct}}u(t,x)\Big)= 1  & \hbox{for every $0\le c<c^*$},\vspace{3pt}\\
 \displaystyle\lim_{t\rightarrow +\infty}\Big(\sup_{x\in \R^N\setminus B_{ct}}u(t,x)\Big)= 0 & \hbox{for every $c>c^*$},
\end{array}
\right.
\end{eqnarray*}
where $B_r$ stands for the open Euclidean ball with center $0$ and radius $r>0$.

For the spatially periodic equation~\eqref{main}, the problem is more intricate, even for compactly supported non-trivial initial data. Through probabilistic techniques, Freidlin and G\"{a}rtner \cite{FG} extended the result of~\cite{AW} to the periodic case, under the Fisher-KPP condition: $0< f(x,s)/s\le f_s(x,0)=\frac{\partial f}{\partial s}(x,0)$ for all~$(x,s)\in\R^N\times (0,1)$. Combined with~\cite{BHN,BHN1,W}, their result says that the spreading speed $\omega_0(\xi)$ in any direction $\xi$ exists, is independent of the compactly supported initial data, and can be characterized by
$$\omega_0(\xi)=\inf_{\su{e\in \mathbb{S}^{N-1}}{\xi\.e>0}}\frac{c^*(e)}{\xi\cdot e},$$
where the subscript $0$ in $\omega_0(\xi)$ refers to the case of compactly supported initial conditions, and where $c^*(e)$ is the minimal speed of pulsating traveling fronts in the direction $e$ (in the sense of Definition~\ref{def:pf} below). Later on, Rossi~\cite{R1} presented a PDE approach of the Freidlin-G\"{a}rtner formula for the spreading speed in a more general framework.

The above two examples show that the spreading speeds are closely related to the speeds of traveling fronts. These fronts, called pulsating traveling fronts in the periodic case, describe the invasion of the steady state $0$ by the steady state $1$ in a direction $e$:

\begin{definition}\label{def:pf}
A \PTF\ connecting~$1$ to~$0$ in the direction $e$ with a speed $c^*(e)$ is an entire $($defined for all $t\in\R$$)$ classical solution $\phi_e:\R\times\R^N\to(0,1)$ of~\eqref{main} of the type
$$\phi_e(t,x)=U_e(x,x\cdot e-c^*(e)t),$$
where the function $(x,z)\in\R^N\times\R \mapsto U_e(x,z)$ is periodic in the $x$-variable and satisfies
$$\lim_{z\to-\infty}U_e(x,z)=1,\ \ \lim_{z\to+\infty}U_e(x,z)=0,\ \hbox{uniformly with respect to $x\in\R^N$}.$$
\end{definition}

In this paper, we mainly consider three types of $f$, always assuming~\eqref{f01}. The first type means that the steady state $0$ is unstable and $1$ is stable. We also assume the KPP condition for the first type, as in the Freidlin-G{\"a}rtner paper~\cite{FG}:

\begin{hypothesis}\label{hyp:KPP}
For every $(x,s)\in\R^N\times(0,1)$, $0<f(x,s)\le f_s(x,0) s$.
\end{hypothesis}

The second type is the generalized ignition case: 

\begin{hypothesis}\label{hyp:ignition}
There exist $0<\sigma<\varsigma<1$ such that $f=0$ in $\R^N\times [0,\sigma]$, $f\ge0$ in~$\R^N\times[\sigma,1]$, $\max_{\R^N} f(\cdot,s)>0$ for every $s\in (\sigma,1)$, $f>0$ in $\R^N\times(\varsigma,1)$, and $f(x,\cdot)$ is non-increasing in $[\varsigma,1]$ for every $x\in\R^N$.
\end{hypothesis}

The third type means that the steady states $0$ and $1$ are weakly stable and there exists a pulsating front connecting  $1$ to $0$ with a positive speed in each direction:

\begin{hypothesis}\label{hyp:bistable}
There exists $\sigma\in(0,1/2)$ such that $f(x,\cdot)$ is non-increasing in $[0, \sigma]$ and $[1-\sigma, 1]$ for every $x\in\R^N$, and $\min_{\R^N} f(\cdot,s)< 0$ for every $s\in(0,\sigma]$. Moreover, for every $e\in\mathbb{S}^{N-1}$, there is a pulsating front $\phi_e(t,x)=U_e(x,x\cdot e-c^*(e)t)$ connecting~$1$ to~$0$ with $c^*(e)>0$.
\end{hypothesis}

By~\cite{BH1,BHN, W}, Hypothesis~\ref{hyp:KPP} ensures that, for every $e\in\mathbb{S}^{N-1}$, there is a minimal speed $c^*(e)>0$  such that~\eqref{main} admits pulsating fronts $U_e(x,x\cdot e-c t)$ connecting $1$ to $0$ if and only if $c\ge c^*(e)$. Moreover, $c^*(e)$ can be expressed in terms of the periodic principal eigenvalues of some associated linear operators, from which one deduces that $c^*(e)$ is continuous with respect to $e\in \mathbb{S}^{N-1}$~\cite{BHN1}. By \cite{BH1,X1}, Hypothesis~\ref{hyp:ignition} guarantees that, for every $e\in\mathbb{S}^{N-1}$, there exist a unique speed $c^*(e)>0$ and a unique (up to translation of $z$) profile $U_e(x,z)$. From~\cite[Proposition~2.2]{GHR1}, if Hypothesis~\ref{hyp:ignition} or~\ref{hyp:bistable} holds, then $c^*(e)$ is unique, $U_e$ is unique up to shift in $z$ and the map $e\mapsto c^*(e)$ is continuous. For the homogeneous case~\eqref{HRD} with $f$ of bistable type:
$$\exists\,\beta\in(0,1),\ \ f<0\hbox{ in }(0,\beta)\hbox{ and }f>0\hbox{ in }(\beta,1),$$
with $\int_0^1 f>0$, then~\eqref{HRD} admits planar fronts connecting $1$ to $0$ with positive and unique speed $c$, namely solutions $u(t,x)=\psi(x\cdot e-ct)$, with $e\in\mathbb{S}^{N-1}$ and $\psi: \R\to (0,1)$ such that $\psi(-\infty)=1$ and $\psi(+\infty)=0$, see \cite{AW,FM}. Such fronts are pulsating traveling fronts too, with the same speed $c$. Furthermore, if $f$ is non-increasing in a right neighborhood of~$0$ and in a left neighborhood of $1$ (for instance if $f'(0)<0$ and  $f'(1)<0$), then Hypothesis~\ref{hyp:bistable} holds, and any pulsating traveling front in the direction $e$ is a planar front by~\cite[Theorem~3.1]{BH2}. Planar traveling fronts can also exist for multistable nonlinearities~$f$, see~\cite{DLL,FM}. For the periodic equation~\eqref{main}, if $f$ is of the strong bistable type, namely if
$$\partial_sf(x,0)<0\hbox{ and $\partial_sf(x,1)<0$ for all $x\in\R^N$},$$
then Hypothesis~\ref{hyp:bistable} is known to hold in dimension~$1$~\cite{DHZ2,FZ,Z2} or in higher dimensions under various additional assumptions on $A$, $q$ and $f$, for instance in highly or slowly oscillating media~\cite{DLL,DS,X4} or when there does not exist any stable periodic steady state between $0$ and $1$~\cite{DG,D2,GR} (see also~\cite{HPS,PX,VV,X2,X3,XZ} for further references in the case of almost-homogeneous coefficients).

Lastly, any of Hypotheses~\ref{hyp:KPP}-\ref{hyp:bistable} guarantees that the invasion of $0$ by $1$ happens for every solution with a not-too-small initial datum $u_0$ such that $0\le u_0\le 1$, see \cite[Theorem 3 and Corollary 1]{DR} and \cite[Proposition~2.3]{GHR1}, in the following sense:

\begin{proposition}[\cite{DR,GHR1}]\label{pro:rho}
Assume that one of Hypotheses~$\ref{hyp:KPP}$-$\ref{hyp:bistable}$ holds. Then there exist $\theta\in (0,1)$ and $\rho>0$ such that  for any $x_0\in\R^N$, if 
$$\theta\,\1_{B_{\rho}(x_0)}\le u_0\le 1\ \hbox{ in $\R^N$},$$
where $B_{\rho}(x_0)$ denotes the open Euclidean ball of center $x_0$ and radius $\rho$, then the solution $u(t,x)$ of \eqref{main} with initial condition $u_0$ satisfies $u(t,x)\rightarrow 1$ as $t\rightarrow +\infty$ locally uniformly with respect to $x\in\R^N$. Such a solution is then said to be {\rm invading}.
\end{proposition}

Furthermore, it follows from \cite{FG} that, if $q=0$ and Hypothesis~\ref{hyp:KPP} holds, then $\theta\in(0,1)$ and $\rho>0$ can be arbitrary. Under any one of Hypotheses~\ref{hyp:KPP}-\ref{hyp:bistable}, one knows from~\cite{R1} that the spreading speeds and the spreading set of invading solutions with {\it compactly supported} initial data exist and are given by
\begin{equation}\label{w0}
\omega_0(\xi)=\min_{\su{e\in \mathbb{S}^{N-1}}{e\.\xi>0}}\frac{c^*(e)}{e\.\xi}\;\in (0,+\infty),
\end{equation}
and 
\begin{equation}\label{W0}
\W_0=\Big\{r\xi\,:\,\xi\in \mathbb{S}^{N-1},\ \ 0\leq r<\omega_0(\xi) \Big\},
\end{equation}
respectively.\footnote{It is assumed in addition in \cite{R1} that $s\mapsto f(x,s)$ is nonincreasing in a left neighborhood of $1$. Such an assumption is not supposed in the present paper in the KPP case (Hypothesis~\ref{hyp:KPP}). However, one can get rid of it by considering two KPP nonlinearities $\ul f,\ol f$ satisfying $\ul f\leq f\leq \ol f$ and $\partial_s\ul f(\.,0)\equiv\partial_s f(\.,0)\equiv\partial_s \ol f(\.,0)$ for which this additional assumption holds, and then deduce the result for $f$ by comparison, since the minimal speeds $c^*(e)$ only depend on the reaction term through its linearization at $0$, according to \cite{BH1,BHN1,W}.} Clearly, since $c^*(e)>0$, \eqref{W0} is equivalent to the following formulation
\be\label{W0-wulff}
\W_0=\bigcap_{e\in\Sph}\!\!\mc{H}_e,\ \ \hbox{with }\mc{H}_e:=\big\{x\in\R^N:x\cdot e<c^*(e)\big\},
\ee
which is the expression of the Wulff shape arising in crystallography. One deduces from this expression that $\W_0$ is convex. One also sees from~\eqref{w0} and the positivity and continuity of $c^*(e)$ with respect to $e$, that the map $\xi\mapsto \omega_0(\xi)$ is continuous in $\mathbb{S}^{N-1}$.

%%%%%%%%%%%%%%%%%%%%%%%%%%%%%%%%%%%%%%%%%%%%%%%%%%%

\subsection{Main results}

In the following, $u$ is a solution to the equation \eqref{main} with an initial datum of the form~${u_0=\1_{U}}$, where $U$ is an unbounded Borel subset of $\R^N$ (the results actually hold as well when $U$ is bounded, but they are not new in that case). In the present paper, we both extend the results of~\cite{BHN,BHN1,FG,R1,W} to the case of general unbounded initial supports, as well as those of~\cite{HR1} to the periodic setting. This extension requires the understanding of the interplay between the initial support~$U$ and the heterogeneity of~\eqref{main}. Furthermore, as shall be seen, new unexpected phenomena different from the homogeneous equation~\eqref{HRD} arise for the periodic equation~\eqref{main}. They are related to the possible lack of smoothness of the spreading set $\W_0$ for solutions with compactly supported initial data, which can actually occur, at least in the bistable case, cf.~\cite{DG} and \cite[Corollary~6.3]{GR}. If the set $\W_0$ is differentiable at any boundary point, then,  as seen in Proposition~\ref{pro:suf-conditions} below, for every direction $e\in\Sph$, the hyperplane
$$H_e:=\partial\mc{H}_e=\{x\in\R^N:x\cdot e=c^*(e)\}$$
is a supporting hyperplane for $\W_0$ (see Figure~\ref{fig:Supt-H}),\footnote{An affine hyperplane $H\subset\R^N$ is called a supporting hyperplane of set $A\subset\R^N$ if $A\setminus H$ is connected (that is, $A$ lies on one closed side of $H$) and $H\cap\overline{A}\neq\emptyset$.} which equivalently means that every $e\in\mathbb{S}^{N-1}$ is a minimizer of the Freidlin-G\"{a}rtner formula \eqref{w0}, for some $\xi\in\Sph$, i.e.,
\be\label{supt-plane}
\forall\,e\in\mathbb{S}^{N-1},\, \exists\,\xi\in\mathbb{S}^{N-1},\quad \frac{c^*(e)}{e\.\xi}=\omega_0(\xi).
\ee
We also derive another sufficient condition for \eqref{supt-plane}. Both are given by the following.
\begin{figure}[ht]
\centering
{\includegraphics[width=.55\linewidth]{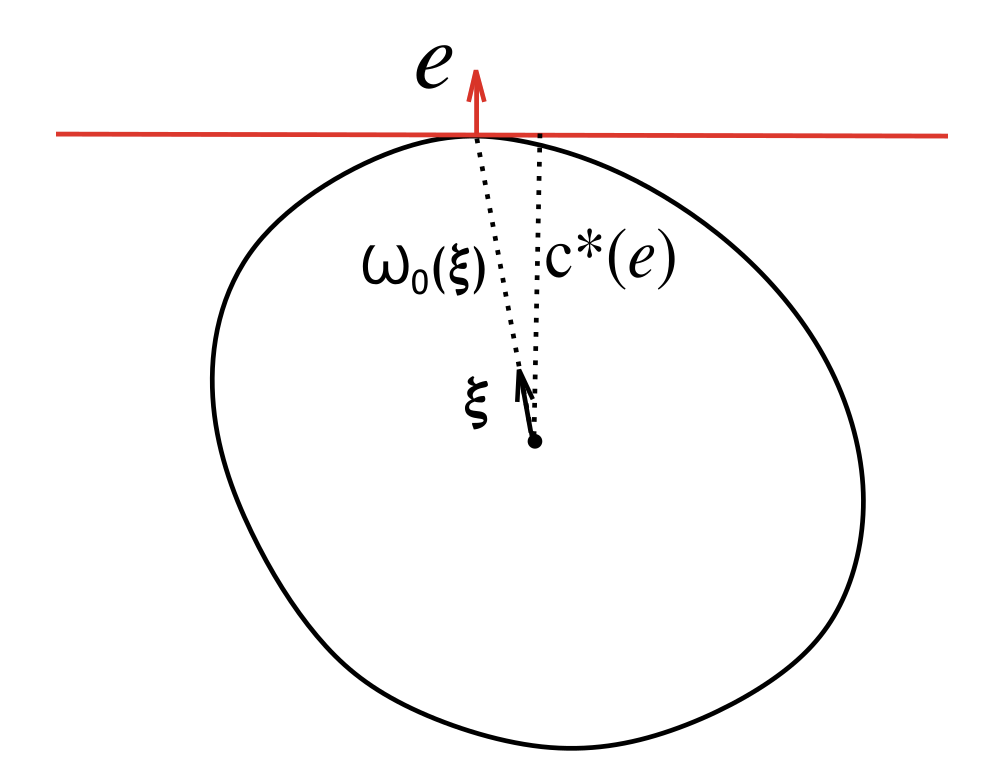}}
\caption{Supporting hyperplane $H_e$ for $\W_0$.}
\label{fig:Supt-H}
\end{figure}

\begin{proposition}\label{pro:suf-conditions}
Assume one of Hypotheses~\ref{hyp:KPP}-\ref{hyp:bistable}. Any of the following conditions ensures~\eqref{supt-plane}:
\begin{itemize}
\item[(i)] the spreading set $\W_0$ defined by~\eqref{w0}-\eqref{W0-wulff} is differentiable at any boundary point;
\item[(ii)] for any $\xi\in\Sph$, there exists a unique $e\in\Sph$ such that $e\cdot\xi>0$ and
$$\frac{c^*(e)}{e\.\xi}=\min_{\su{e'\in\Sph}{e'\.\xi>0}}\frac{c^*(e')}{e'\.\xi}=\omega_0(\xi).$$
\end{itemize}
\end{proposition}

Our main results are derived under the assumption \eqref{supt-plane}. If such assumption fails for some direction $e$, it loosely means that the half-space $\mc{H}_e$  is not relevant in the Wulff shape formulation \eqref{W0-wulff} of the set $\W_0$. This is the reason behind the fact that some results which hold for the homogeneous equation~\eqref{HRD} do not extend to the heterogeneous setting~\eqref{main}, see Remark \ref {remCE1}.
 
In order to state our main results, we introduce some notation.
For given $x\in\R^N$ and $E\subset\R^N$, 
we set $d(x,E):=\inf\{|x-y|:\, y\in E\}$, where $|\cdot|$ is the Euclidean norm, with the convention
$d(x,\emptyset):=+\infty$. 
We then define the notion of positive-distance-interior of a set $U\subset\R^N$, as
$$U_{\rho}:=\Big\{x\in U:\, d(x,\partial U)\ge \rho \Big\}, \hbox{ with $\rho>0$}.$$
Next, for any pair of subsets $E$ and $F$ of $\R^N$, we denote
$$d_{H}(E,F):=\max\Big(\sup_{x\in E}d(x,F),\sup_{y\in F}d(y,E)\Big),$$
the Hausdorff distance between $E$ and $F$, with the convention that it is $+\infty$ if $E$ or $F$ is empty.
Finally, for $\mc{A}\subset\Sph$, we call $\mc{C}(\mc{A})$ the cone generated by $\mc{A}$, defined by
\be\label{defCA}
\mc{C}(\mc{A}):=\big\{ra:\ r\ge0,\ a\in \mc{A}\big\},
\ee
with the convention $\mc{C}(\emptyset):=\{0\}$.

Our first result concerns the long-time description of the upper level sets of the solution.

\begin{theorem}\label{Th:levelset}
Assume one of Hypotheses~\ref{hyp:KPP}-\ref{hyp:bistable}, and that~\eqref{supt-plane} holds. If $U_{\rho}\neq \emptyset$, with $\rho>0$ given in Proposition~\ref{pro:rho}, and moreover
\Fi{UUrho}
d_H(U,U_\rho)<+\infty,
\Ff
then, for any $\lambda\in(0,1)$, there holds
$$\frac1t\;d_H\big(\{x\in\R^N:u(t,x)>\lambda\}\,,\, U+t\W_0\big)
\mathop{\longrightarrow}_{t\to+\infty}0.$$
\end{theorem}

One can intuitively understand the above result in the following way. If the initial support $U$ is a Euclidean ball with radius $\rho>0$ given by Proposition~\ref{pro:rho}, then the region where $u$ is close to $1$ is approximately given by the set $t\mathcal{W}_0$, for large $t$. For an unbounded initial support~$U$, one can take balls $B_{\rho}(x_0)$ centered at any given $x_0\in U_{\rho}$ as initial supports, and infer, by comparison, that the region where $u\sim 1$ approximately grows {\em at least} as $U_\rho+t\W_0$. Clearly, assumption \eqref{UUrho} allows one to replace $U_\rho$ by~$U$. The fact that $u\sim 1$ also grows {\em at most} as $U+t\W_0$ is the hardest part of the proof. This interpretation is analogous to the Huygens principle in optics, according to which every point on a wavefront acts as a source of secondary wavelets. We show in Remark~\ref{remCE1} below that, without the assumption~\eqref{supt-plane}, Theorem~\ref{Th:levelset} may fail, contradicting the Huygens principle.

Next, we focus on the spreading set. From the expressions \eqref{w0}-\eqref{W0-wulff}, the spreading speed and set for compact initial supports are independent of the shapes of supports. However, for general unbounded initial supports, these two notions of spreading speed and set are much more intricate and rely on bounded and unbounded directions of the initial supports~$U$. This is shown by the last two authors of the present paper in~\cite{HR1} for the homogeneous case~\eqref{HRD}. Loosely speaking, in order to obtain the spreading set for a general initial support $U$, one should be able to take the limit as $t\to+\infty$ of the set $U+t\W_0$, provided by Theorem~\ref{Th:levelset}, rescaled by $1/t$. This is not always possible, and when it is not, the spreading set may not exist, see~\cite[Section~6]{HR1}. A sufficient condition that allows one to pass to the limit is expressed in terms of the set of {\it bounded directions} and the set of {\it unbounded directions} of $U\subset\R^N$, which are defined respectively by
$$\left\{\baa{l}
\displaystyle\mc{B}(U):=\Big\{\xi\in\Sph:\liminf_{\tau\to+\infty}\frac{d(\tau\xi,U)}{\tau}>0\Big\},\vspace{3pt}\\
\displaystyle\mc{U}(U):=\Big\{\xi\in\Sph:\lim_{\tau\to+\infty}\frac{d(\tau\xi,U)}{\tau}=0\Big\}.\eaa\right.$$
The sets~$\mc{B}(U)$ and $\mc{U}(U)$ are disjoint and respectively open and closed relatively to $\Sph$. In general, there may exist a direction that it is neither bounded nor unbounded. For instance, for any given $e\in\mathbb{S}^{N-1}$, if $U=B_1+\{2^n e:\, n\in\mathbb{N}\}$, then $e\not\in \mathcal{B}(U)\cup \mathcal{U}(U)$. We refer to~\cite{HR1} for more details and further examples. 

\begin{theorem}\label{Th:Ass}
Assume one of Hypotheses~\ref{hyp:KPP}-\ref{hyp:bistable} and  that~\eqref{supt-plane} holds. 
If $U_{\rho}\neq \emptyset$, with $\rho>0$ given in Proposition~\ref{pro:rho}, and if
\begin{equation}\label{BUS}
\mc{B}(U)\cup \mc{U}(U_{\rho})=\mathbb{S}^{N-1},
\end{equation}
then the~spreading set  for~$u$ exists and is given by 
$$\W=\mc{C}(\mc{U}(U))+\W_0.$$
\end{theorem}

\begin{remark}\label{remCE1}{\rm Without the assumption~\eqref{supt-plane} the conclusions of Theorems~\ref{Th:levelset}-\ref{Th:Ass} may fail. In particular, if~\eqref{supt-plane} is not fulfilled then the set $\mc{C}(\mc{U}(U))+\W_0$ may fail being a spreading set, and actually may fail being a spreading superset. Assume indeed that there is $e\in\Sph$ for which~\eqref{supt-plane} does not hold,~i.e.
$$R_e:=\max_{x\in \overline{\W_0}} \;x\cdot e<c^*(e)$$
(such a situation can actually occur, as shown in~\cite{DG,GR}), which means that the hyperplane~$\partial\mc{H}_e=\{x\in\R^N:x\cdot e=c^*(e)\}$ is not a supporting hyperplane for $\W_0$. If one chooses $U=\{x\in\R^N:\, x\cdot e\le 0\}$, it follows from~\cite{LZ1,W} under Hypothesis~\ref{hyp:KPP}, and from \cite{X1,X2} under Hypotheses~\ref{hyp:ignition}-\ref{hyp:bistable}, that the spreading set is $\W=\mc{H}_e=\{x\in\R^N:x\cdot e<c^*(e)\}$, which strictly contains the set
$$\mc{C}(\mathcal{U}(U))+\W_0=U+\W_0=
\{x\in\R^N:\ x\cdot e<R_e\},$$
whence the conclusion of Theorem \ref{Th:Ass} fails. Moreover, for any $\lambda\in(0,1)$, there holds
$$\baa{ll}
& d_H\big(\{x\in\R^N:u(t,x)>\lambda\}\,,\, U+t\W_0\big)\vspace{3pt}\\
= \ & d_H\big(\{x\in\R^N:\ x\cdot e<c^*(e)t\}\,,\, \{x\in\R^N:\ x\cdot e<R_e t\}\big)+o(t)\vspace{3pt}\\
= \ & (c^*(e)-R_e)t+o(t) \qquad\text{as }\, t\to+\infty,\eaa
$$
which shows that the conclusion of Theorem \ref{Th:levelset} fails too.}
\end{remark}

If $U$ is star-shaped with respect to a point $x_0\in\R^N$ and $d_{H}(U,U_{\rho})<+\infty$,
 then~\eqref{BUS} holds. In particular, if $U$ is a cone, namely if there are $x_0\in\R^N$ and a set $\mc{A}\subset\Sph$ such that $U=x_0+\mc{C}(\mc{A})$, and if $d_{H}(U,U_{\rho})<+\infty$, then~\eqref{BUS} holds. Other conditions ensuring~\eqref{BUS} are given in~\cite[Proposition~5.1]{HR1}.
  
By geometric arguments, we can deduce from Theorem~\ref{Th:Ass} the generalized Freidlin-G\"{a}rtner formula~\eqref{ASS-G} below for the spreading speeds of solutions of~\eqref{main} with general initial supports. To do so, let us introduce an auxiliary notation, for $\zeta\in\R^N\setminus\{0\}$:
$$\hat\zeta:=\frac{\zeta}{|\zeta|}.$$

\begin{theorem}\label{Th:asspeed}
Under the conditions of Theorem~\ref{Th:Ass}, the solution~$u$ admits spreading speed $\omega(\xi)$ for every $\xi\in\mathbb{S}^{N-1}$, and it is given by
\begin{eqnarray}\label{ASS-G}
\omega(\xi)=\left\{\begin{array}{ll}
\displaystyle\sup_{z\in\mc{C}(\mc{U}(U))}\frac{\omega_0(\widehat{\xi-z})}{|\xi-z|}\ \in[\omega_0(\xi),+\infty) & \hbox{if $\xi\in\mc{B}(U)$},\\
+\infty & \hbox{if $\xi\in \mathcal{U}(U)=\Sph\!\setminus\!\mathcal{B}(U)$},
\end{array}\right.
\end{eqnarray}
and the spreading set $\W$ can be rewritten as 
$$\W=\big\{r\xi:\, \xi\in\mathbb{S}^{N-1},\, 0\le r< \omega(\xi)\big\}.$$
\end{theorem}

From the closedness of $\mc{U}(U)$ and $\mc{C}(\mathcal{U}(U))$ and the continuity and positivity of~$\omega_0$ in~$\Sph$, it follows that the $\sup$ in \eqref{ASS-G} when $\xi\in\mathcal{B}(U)$ is reached and positive, and that the map $\xi\mapsto \omega(\xi)\in (0,+\infty]$ is continuous in $\mathbb{S}^{N-1}$.

Furthermore, on the one hand, if in Theorem~\ref{Th:asspeed} the set $U$ is bounded, then $\mc{C}(\mc{U}(U))=\mc{C}(\emptyset)=\{0\}$ and~\eqref{ASS-G} amounts to the standard Freidlin-G\"artner formula $\omega(\xi)=\omega_0(\xi)$ for each $\xi\in\Sph$. On the other hand, for the homogeneous equation~\eqref{HRD}, for which $\omega_0(\xi)=c^*>0$ is independent of~$\xi$,~\eqref{ASS-G} reduces by direct inspection to 
$$\omega(\xi)=\left\{\baa{ll}
\displaystyle\!\sup_{\su{e\in\mathcal{U}(U)}{e\.\xi>0}}\!\frac{c^*}{\sqrt{1-(e\cdot\xi)^2}}\ \in(c^*,+\infty) & \!\!\hbox{if $\xi\in \mathcal{B}(U)$ and $\big\{e\in \mathcal{U}(U):e\cdot\xi>0\big\}\neq\emptyset$},\vspace{3pt}\\
\!c^* & \!\!\hbox{if $\xi\in\mathcal{B}(U)$ and $\big\{e\in \mathcal{U}(U):e\cdot\xi>0\big\}=\emptyset$},\vspace{3pt}\\
\!+\infty & \!\!\hbox{if $\xi\in \mathcal{U}(U)$}.\eaa\right.$$
In other words, Theorem~\ref{Th:asspeed} generalizes \cite[Theorem~2.1]{HR1} to the periodic equation~\eqref{main} and under more general assumptions on $f$.

If the condition~\eqref{BUS} is violated, then there may not exist spreading sets and speeds even for the homogeneous equation~\eqref{HRD}. For example, if $U=\cup_{n\in\mathbb{N}}B_{2^n+1}\!\setminus\!B_{2^n-1}$ and $f(s)=s(1-s)$, then $\mathcal{B}(U)=\mathcal{U}(U)=\emptyset$ and no spreading sets and speeds exist, see  \cite[Proposition~6.1]{HR1} and see \cite{HR1} also for other counter-examples.

Finally, we show equivalent formulas of the spreading speeds and sets for particular cone-like supports, in which we will see the relationship between the spreading sets and the hyperplanes $H_e=\{x\in\R^N:x\cdot e=c^*(e)\}$. Namely, let us now consider the case when either~$\mathcal{B}(U)$ or~$\mathcal{U}(U)$ generates a non-trivial convex cone. If $\emptyset\neq\mc{B}(U)\subsetneqq\Sph$ and~$\mc{C}(\mathcal{B}(U))$ is convex, then $U$ is called V-shaped. If $\emptyset\neq\mc{U}(U)\subsetneqq\Sph$ and $\mc{C}(\mathcal{U}(U))$ is convex, then $U$ is called $\Lambda$-shaped. For $\mc{A}\subset\Sph$, we define
\be\label{defLA}
\mc{L}(\mc{A}):=\big\{e\in\Sph:\ \xi\cdot e\ge0 \text{ \ for all \ }\xi\in \mc{A}\big\}.
\ee
Notice that $\mc{L}(\mc{A})$ is always closed in $\Sph$, and $\mc{C}(\mc{L}(\mc{A}))$ is always convex. If~$U$ is V-shaped, then $\emptyset\neq\mc{L}(\mc{B}(U))\subsetneqq\Sph$. If $U$ is $\Lambda$-shaped, then ${\emptyset\neq\mc{L}(\mc{U}(U))\subsetneqq\Sph}$.

\begin{corollary}\label{theo:ass}
Assume that all conditions of Theorem~\ref{Th:Ass} hold.
\begin{enumerate}
\item[(i)] If the initial support $U$ is V-shaped, then the spreading set for $u$ is given by
$$\mathcal{W}_V:=\Big\{x\in\R^N:\!\displaystyle\inf_{e\in\mathcal{L}(\mathcal{B}(U))}(x\cdot e- c^*(e))<0\Big\}=\big\{r\xi:\xi\in\mathbb{S}^{N-1},\, 0\le r< \omega_V(\xi)\big\},$$
where 
\begin{eqnarray}\label{ASS-V}
\omega_V(\xi):=\left\{\begin{array}{lll}
\displaystyle\sup_{e\in\mathcal{L}(\mathcal{B}(U))} \frac{c^*(e)}{e\cdot \xi}\in (0,+\infty) & \hbox{if $\xi\in \mathcal{B}(U)$},\\
+\infty & \hbox{if $\xi\in\mc{U}(U)=\mathbb{S}^{N-1}\!\setminus\!\mathcal{B}(U)$}.
\end{array}
\right.
\end{eqnarray}
\item[(ii)] If the initial support $U$ is $\Lambda$-shaped, then the spreading set for $u$ is given by
$$\mathcal{W}_{\Lambda}:=\Big\{x\in\R^N:\!\displaystyle\inf_{e\in\mathcal{L}(\mathcal{U}(U))}(x\cdot e+ c^*(-e))> 0\Big\}=\big\{r\xi:\xi\in\mathbb{S}^{N-1},\, 0\le r<\omega_{\Lambda}(\xi)\big\},$$
where
\begin{eqnarray}\label{omega-Lambda}
\omega_{\Lambda}(\xi):=\left\{\begin{array}{ll}
\displaystyle\inf_{\su{e\in\mathcal{L}( \mathcal{U}(U))}{e\.\xi<0}}\frac{c^*(-e)}{-e\cdot \xi}\in (0,+\infty) & \hbox{if $\xi\in\mc{B}(U)$},\\
+\infty &  \hbox{if $\xi\in\mathcal{U}(U)=\mathbb{S}^{N-1}\!\setminus\!\mathcal{B}(U)$}.
\end{array}
\right.
\end{eqnarray}
\end{enumerate}
\end{corollary}

In the case when $U$ is V-shaped, we have from Lemma~\ref{lemma:P} below that
$$\mc{C}(\overline{\mathcal{B}(U)})=\big\{x\in\R^N:\forall\,e\in\mathcal{L}(\mathcal{B}(U)),\,x\cdot e\ge 0\big\}.$$
It then follows from the relative openness of $\mathcal{B}(U)$ and closedness of $\mathcal{L}(\mathcal{B}(U))$ in $\Sph$ that $\inf_{e\in\mathcal{L}(\mathcal{B}(U))} \{e\cdot\xi\}>0$ for any $\xi\in \mathcal{B}(U)$ and that $\inf_{e\in\mathcal{L}(\mathcal{B}(U))} \{e\cdot\xi\}=0$ for any $\xi\in \overline{\mathcal{B}(U)}\!\setminus\!\mc{B}(U)$. Therefore, by continuity and positivity of the map $e\mapsto c^*(e)$ in~$\Sph$, the $\sup$ in $\omega_V(\xi)$ is finite and reached for any $\xi\in\mc{B}(U)$, and the map $\xi\mapsto\omega_V(\xi)\in(0,+\infty]$ is continuous in~$\mathbb{S}^{N-1}$. Similarly, when $U$ is $\Lambda$-shaped, the set $\{e\in\mc{L}(\mc{U}(U)):e\cdot\xi<0\}$ is not empty and the $\inf$ in $\omega_{\Lambda}(\xi)$ is reached for any $\xi\in\mc{B}(U)$, and the map~$\xi\mapsto\omega_{\Lambda}(\xi)\in(0,+\infty]$ is continuous in $\mathbb{S}^{N-1}$. 

\begin{remark}\label{remCE}{\rm
Corollary~\ref{theo:ass} shows some equivalent expressions of the spreading set~$\mc{C}(\mathcal{U}(U))+\W_0$ for V-shaped and $\Lambda$-shaped initial supports $U$, under the assumptions~\eqref{supt-plane} and~\eqref{BUS}. In the proof of Corollary~\ref{theo:ass}, we show that for a V-shaped initial support~$U$ satisfying~\eqref{BUS}, even if $\W_0$ does not satisfy condition~\eqref{supt-plane}, there is another expression of $\mc{C}(\mathcal{U}(U))+\W_0$, namely
$$\mc{C}(\mathcal{U}(U))+\W_0=\big\{r\xi:\, \xi\in\mathbb{S}^{N-1},\, 0\le r<\hat{\omega}(\xi)\big\},$$
where
\be\label{hatomega}
\hat{\omega}(\xi):=\left\{\begin{array}{ll}
\displaystyle\sup_{e\in\mathcal{L}(\mathcal{B}(U))} \frac{\hat{c}(e)}{e\cdot \xi}\in (0,+\infty) & \hbox{if $\xi\in \mathcal{B}(U)$},\\
+\infty & \hbox{if $\xi\in\mc{U}(U)=\mathbb{S}^{N-1}\setminus \mathcal{B}(U)$},\end{array}\right.
\ee
with $\hat{c}$ being the support function of $\W_0$, i.e.
$$\hat{c}(e):=\sup_{\xi\in\mathbb{S}^{N-1}} \omega_0(\xi)\xi\cdot e\le c^*(e),$$
see~\eqref{hatW2} below. If~\eqref{supt-plane} holds, then $\hat{c}(e)=c^*(e)$ for all $e\in\Sph$, and $\hat{\omega}(\xi)= \omega_{V}(\xi)$ for all $\xi\in\Sph$. 
Instead, if~\eqref{supt-plane} does not hold, and moreover 
$\hat{\omega}(\xi)< \omega_{V}(\xi)$ for some $\xi\in\mathcal{B}(U)$, then $\mc{C}(\mathcal{U}(U))+\W_0\subsetneqq\W_V$. 
We recall indeed the example in Remark~\ref{remCE1}~with
$$U=\big\{x\in\R^N:\, x\cdot e\le 0\big\},$$
which is V-shaped $($and also $\Lambda$-shaped$)$ and satisfies~\eqref{BUS},
for which one has $\W=\mc{H}_e:=\{x\in\R^N:\,x\cdot e< c^*(e)\}$, which strictly contains  
$\mc{C}(\mathcal{U}(U))+\W_0=\{x\in\R^N:\,x\cdot e<\hat{c}(e)\}$ provided that~$\hat{c}(e)<c^*(e)$.
This is a counter-example to Corollary~\ref{theo:ass} without~\eqref{supt-plane} for both $V$-shaped and $\Lambda$-shaped initial supports.}
\end{remark}

\subsubsection*{Strategy of the proofs and outline of the paper} 

Our main results are proved by deriving some lower and upper bounds on the upper-level sets of the solution. While the lower bounds are readily deduced from Proposition~\ref{pro:rho} via a superposition principle, the proofs of the upper bounds require new ideas. The cornerstone is the use of a family of auxiliary solutions which are initially supported on large dilations of suitable approximations of the set~$-\W_0$. We call them {\em retracting solutions}, because their level sets retract as time progresses. The problem then boils down to locating these level sets. This is where we will need assumption~\eqref{supt-plane}, without which our results fail in general, as shown in Remark~\ref{remCE1}.

More precisely, the paper is organized as follows. In Section~\ref{subsection2.1} we prove Proposition~\ref{pro:suf-conditions}, which provides one with two sufficient conditions for the validity of the key assumption~\eqref{supt-plane}. Section~\ref{subsection2.2} contains some auxiliary geometric results, especially on convex cones and on the approximation of compact convex sets by polyhedral compact sets, which will be used in the construction of the retracting solutions. Section~\ref{subsection2.3} contains some Liouville-type properties of entire solutions with super-critical speed of propagation in the ignition or weakly bistable cases (Hypotheses~\ref{hyp:ignition} or~\ref{hyp:bistable}). Section~\ref{section3} contains the derivation of some lower bounds of the upper-level sets of the solutions, stated in Proposition~\ref{prop:subset}. Section~\ref{section4} is devoted to the derivation of upper bounds of the upper-level sets of the solutions of~\eqref{main}, stated in Proposition~\ref{prop:superset} in Section~\ref{subsection4.2}. The core of the proof of the upper bounds, in Section~\ref{subsection4.1}, is itself based on some quantitative upper estimates of the retracting solutions. These upper estimates use the application of parabolic maximum principles in suitable dilated domains, and some properties of eigenvalues of elliptic operators obtained after linearizing the equation. The proofs actually rely on different arguments depending on whether Hypothesis~\ref{hyp:KPP} or Hypotheses~\ref{hyp:ignition}-\ref{hyp:bistable} are made: finite sum of exponential functions based on polyhedral exterior approximations of $\overline{\mc{W}_0}$ are used in the former case, while in the latter an argument by contradiction in the spirit of \cite{R1} is used, together with $C^1$ exterior approximations of $\overline{\mc{W}_0}$ and comparison and Liouville-type results of Section~\ref{subsection2.3}. The lower and upper estimates of Sections~\ref{section3} and~\ref{section4} will then be used in the proofs of the main results, which are carried out in Section~\ref{section5}. 

%%%%%%%%%%%%%%%%%%%%%%%%%%%%%%%%%%%%%%%%%%%%%%%%%%%
%%%%%%%%%%%%%%%%%%%%%%%%%%%%%%%%%%%%%%%%%%%%%%%%%%%

\section{Preliminaries}\label{section2}

In Section~\ref{subsection2.1}, we prove Proposition \ref{pro:suf-conditions}. We then investigate in Sections~\ref{subsection2.2}-\ref{subsection2.3} some auxiliary geometric, comparison and Liouville-type results, which will be used in the proofs of the lower and upper estimates of the solutions of~\eqref{main} in Sections~\ref{section3}-\ref{section4}.

%%%%%%%%%%%%%%%%%%%%%%%%%%%%%%%%%%%%%%%%%%%%%%%%%%%

\subsection{Proof of Proposition~\ref{pro:suf-conditions}}\label{subsection2.1}

Let us preliminarily observe that, owing to the definition \eqref{W0-wulff} and the continuity of~$c^*$, one has
$$\partial\W_0=\Big\{x\in\R^N\,:\,\max_{e\in\mathbb{S}^{N-1}}(x\.e-c^*(e))=0\Big\}.$$

\begin{proof}[Proof of Proposition~\ref{pro:suf-conditions}] We will show that (i) $\Rightarrow$ (ii) $\Rightarrow$ \eqref{supt-plane}.

For (i) $\Rightarrow$ (ii), take any $\xi\in \mathbb{S}^{N-1}$. Since $\partial \W_0$ is differentiable at the boundary point $\omega_0(\xi)\xi$, there is a unique exterior normal $\nu\in \mathbb{S}^{N-1}$ at $\omega_0(\xi)\xi$. In other words, the hyperplane $\{x\in\R^N:\, (x-\omega_0(\xi)\xi)\cdot \nu=0\}$ is the unique supporting hyperplane at $\omega_0(\xi)\xi$. On the other hand, by~\eqref{w0} and~\eqref{W0-wulff}, there is $e\in\mathbb{S}^{N-1}$ such that $\omega_0(\xi)\xi\cdot e=c^*(e)$ and $x\cdot e<c^*(e)$ for all $x\in\W_0$. Then, by the uniqueness of the supporting hyperplane at $\omega_0(\xi)\xi$, one has $\nu=e$ and $\omega_0(\xi)\xi\cdot \nu=c^*(\nu)$. Therefore, the minimizer $e$ in \eqref{w0} is unique, showing (ii).

For  (ii) $\Rightarrow$ \eqref{supt-plane}, assume that \eqref{supt-plane} does not hold. Then, there exists $e\in\Sph$ such that
\be\label{under-supt-plane}
\forall\,x\in \partial\W_0,\quad x\. e<c^*(e),
\ee
Notice that \eqref{under-supt-plane} also holds  for all $x\in\W_0$, by~\eqref{W0-wulff}. We then set 
$$c:=\max_{x\in \overline{\W_0}}\,x\.e<c^*(e).$$
Let $z\in \overline{\W_0}$ be such that $z\.e=c$. One has that $z\in\partial \W_0$. By \eqref{w0}, there exists $e'\in\Sph$ such that $z\.e'=c^*(e')>0$, that is, $e'$ is a minimizer in~\eqref{w0} for the direction $\hat{z}=z/|z|$. Recalling that $z\.e=c<c^*(e)$, we see that $e'\neq e$, that is, $e'\. e<1$. One can also observe that the hyperplanes $\{x\in\R^N:x\cdot e=c\}$ and $\{x\in\R^N:x\cdot e'=c^*(e')\}$ are two different supporting hyperplanes of $\W_0$ containing the point $z\in\partial\W_0$, which implies in particular that $\partial\W_0$ is not differentiable at $z$, as seen above in the proof of (i)~$\Rightarrow$~(ii).

We then consider, for $n\in\N\setminus\{0\}$, $x_n:=z+\frac1n(e-e')$. We find that
$$x_n\. e=c+\frac{1-e'\.e}n>c,$$
which shows that $x_n\notin\overline{\W_0}$, hence there exists $e_n\in\Sph$ such that $x_n\. e_n>c^*(e_n)$. As a consequence, we get
$$c^*(e_n)<x_n\. e_n=z\.e_n+\frac{e-e'}n\.e_n\leq c^*(e_n)+\frac{e-e'}n\.e_n,$$
which means that $e\. e_n>e'\. e_n$ for all $n\in\N\setminus\{0\}$. From this, we deduce that the sequence $\seq{e}$ has to converge (up to subsequences) towards a limit $e''\neq e'$ (indeed, otherwise, if $e''=e'$, then $e\cdot e'\ge e'\cdot e'=1$, a contradiction). Using again the continuity of $c^*(e)$ and passing to the limit as $n\to+\infty$ in $c^*(e_n)<x_n\cdot e_n$, we finally infer that
$$c^*(e'')\leq z\.e'',$$
and finally $0<c^*(e'')=z\cdot e''$, since $z\in \overline{\W_0}$. Recalling that $\hat{z}=z/|z|$, we have thereby found that both $e'$ and $e''$ satisfy $\hat{z}\cdot e'>0$, $\hat{z}\cdot e''>0$,
$$\frac{c^*(e')}{\hat{z}\.e'}=\frac{c^*(e'')}{\hat{z}\.e''}=|z|=\min_{\su{\t e\in S^{N-1}}{\hat{z}\cdot\t e>0}}\frac{c^*(\t e)}{\hat{z}\cdot\t e},$$
that is, (ii) fails. The proof of Proposition~\ref{pro:suf-conditions} is thereby complete.
\end{proof}

%%%%%%%%%%%%%%%%%%%%%%%%%%%%%%%%%%%%%%%%%%%%%%%%%%%

\subsection{Some auxiliary geometric results}\label{subsection2.2}

We first recall the notations, for any $\mc{A}\subset\Sph$, of the cone $\mc{C}(\mc{A})$ and the set $\mc{L}(\mc{A})\subset\Sph$ defined in~\eqref{defCA} and~\eqref{defLA}, with the convention that $\mc{C}(\emptyset)=\{0\}$. The first two results, Lemmata~\ref{lemma:P}-\ref{lem:distW}, are concerned with an equivalent characterization of $\mc{C}(\overline{\mc{A}})$ when $\mc{C}(\mc{A})$ is convex on the one hand, and with lower estimates of the distance to large dilations of some compact convex sets on the other hand.

\begin{lemma}\label{lemma:P}
For any $\mc{A}\subset\Sph$ such that $\mc{C}(\mc{A})$ is convex, there holds
$$\mathcal{C}(\overline{\mc{A}})=\big\{x\in\R^N:\ x\cdot e\ge 0\text{ \ for all \ } e\in\mathcal{L}(\mc{A})\big\}.$$
\end{lemma}

\begin{proof}
When $\mc{A}=\Sph$, the desired conclusion holds at once, since $\mc{C}(\Sph)=\R^N$ and $\mc{L}(\Sph)=\emptyset$. One can then assume in the sequel that $\mc{A}\subsetneqq\Sph$, and the convexity of~$\mc{C}(\mc{A})$ then implies that $\mc{L}(\mc{A})\neq\emptyset$. 

It is obvious that 
$$\mc{C}(\overline{\mc{A}})\subset \big\{x\in\R^N:\, \forall\,e\in\mathcal{L}(\mc{A}),\ x\cdot e\ge 0\big\}.$$

For the reversed inclusion, assume by contradiction that there is $x_0\in\R^N$ such that
\begin{equation}\label{x0}
\inf_{e\in\mathcal{L}(\mc{A})} \{x_0\cdot e\}\ge 0\ \ \hbox{and}\ \ x_0\not\in\mc{C}(\overline{\mc{A}}).
\end{equation}
Notice that $\mc{C}(\overline{\mc{A}})=\overline{\mc{C}(\mc{A})}$ is closed and non-empty (it always contains the origin~$0$), whence there is $y_0\in\mc{C}(\overline{\mc{A}})$ such that $|x_0-y_0|=d(x_0,\mc{C}(\overline{\mc{A}}))>0$. Since $\mc{C}(\overline{\mc{A}})=\overline{\mc{C}(\mc{A})}$ is closed and convex, the point $y_0$ is the projection of $x_0$ onto $\mc{C}(\overline{\mc{A}})$, whence
\begin{equation}\label{PbP}
\mc{C}(\overline{\mc{A}})\subset \big\{x\in\R^N:(y_0-x_0)\cdot (x-y_0)\ge 0\big\}.
\end{equation}
Then, it follows from the definition of $\mc{C}(\overline{\mc{A}})$ and \eqref{PbP} that 
$$ry_0\in \big\{x\in\R^N:(y_0-x_0)\cdot (x-y_0)\ge 0\big\} \ \ \hbox{for all $r\ge 0$},$$
that is, $(r-1)(y_0-x_0)\cdot y_0\ge 0$ for all $r\ge 0$. As a consequence,
\begin{equation}\label{y0-x0}
(y_0-x_0)\cdot y_0=0.
\end{equation}
Therefore, by following \eqref{PbP} and \eqref{y0-x0},
$$\mc{C}(\mc{A})\subset\mc{C}(\overline{\mc{A}})\subset \big\{x\in\R^N:(y_0-x_0)\cdot x\ge 0\big\},$$
which means that $\hat{y_0-x_0}=(y_0-x_0)/|y_0-x_0|\in\mathcal{L}(\mc{A})$ and $(y_0-x_0)\cdot x_0\ge 0$ by \eqref{x0}. It follows from \eqref{y0-x0} again that  $0\le (y_0-x_0)\cdot x_0=(y_0-x_0)\cdot (x_0-y_0)$, and hence $x_0=y_0$, a contradiction with $x_0\neq y_0$. The proof of Lemma~\ref{lemma:P} is thereby complete.
\end{proof}

\begin{lemma}\label{lem:distW}
Let $K$ be a compact convex set containing a ball $B_\delta$ centered at the origin, with $\delta>0$. Then
$$\forall\,h\geq 1,\quad\min_{x\in K}\ d\big(x,\R^N\!\setminus\!(hK)\big)\geq (h-1)\delta.$$
\end{lemma}

\begin{proof}
For any $e\in\mathbb{S}^{N-1}$, let $\tau(e):=\min_{x\in K}x\cdot e$, which is a well defined real number since $K\neq \emptyset$ is compact. Define
$$K'=\Big\{x\in\R^N:\inf_{e\in\mathbb{S}^{N-1}}(x\cdot e-\tau(e))\ge 0\Big\}.$$
Obviously, $K\subset K'$. Assume now by way of contradiction that $K'\setminus K\neq \emptyset$. Then, for $x_0\in K'\setminus K$, let $y_0\neq x_0$ be the projection of $x_0$ onto the closed convex set $K$. One has
$$\hat{y_0-x_0}\cdot(x-y_0)=\frac{y_0-x_0}{|y_0-x_0|}\cdot (x-y_0)\ge 0\ \hbox{ for all $x\in K$}.$$
Since $x_0\in K'$ and $x_0\neq y_0$, it follows that
$$\hat{y_0-x_0}\cdot x_0-\tau(\hat{y_0-x_0})\ge 0>-|y_0-x_0|=\hat{y_0-x_0}\cdot (x_0-y_0).$$
Hence, for all $x\in K$,
$$\tau(\hat{y_0-x_0})<\hat{y_0-x_0}\cdot y_0\le\hat{y_0-x_0}\cdot x,$$
which contradicts the definition of $\tau(\hat{y_0-x_0})$. Thus, $K=K'$.

Then, for any $h\ge1$,
$$\R^N\setminus(hK)=\Big\{x\in\R^N:\, \inf_{e\in\mathbb{S}^{N-1}}(x\cdot e-h\tau(e))< 0\Big\}.$$
We deduce from $B_\delta\subset K$ that $\tau(e)\leq -\delta$ for all $e\in\Sph$. Then, we get that, for all $x\in K$,
$$\inf_{e\in\mathbb{S}^{N-1}}(x\cdot e-h\tau(e))\ge \inf_{e\in\mathbb{S}^{N-1}}(x\cdot e-\tau(e))+ (h-1)\delta\ge (h-1)\delta.$$
The result follows.
\end{proof}

The last result of this section provides one with two  exterior polyhedral approximations of a given convex compact set~$K$, see Figure~\ref{fig:approximation} below.

\begin{figure}[ht]
\centering
{\includegraphics[width=.45\linewidth]{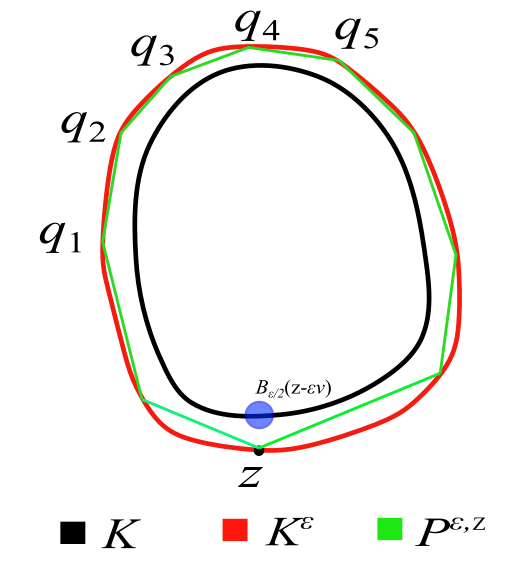}}
\caption{The approximations $K^\e$ and $P^{\e,z}$ of $K$.}
\label{fig:approximation}
\end{figure}

\begin{lemma}\label{lemma:approximation}
Let $K$ be a non-empty convex compact set, and let $\e>0$.
\begin{itemize}
\item[(i)] The set $K^\e:=K+\overline{B_{\e}}\supset K$ is a convex compact set with boundary of class~$C^1$; moreover, $K^\e \rightarrow K$ as $\e\rightarrow 0$, in the sense that $d_H(K^\e,K)\to0$ as $\e\to0$;
\item[(ii)]
There exists $l\in\N$, depending on $K,\e,N$, and a compact polyhedron $P^\e$ with at most $l$ faces such that
\be\label{KPK}
K^{\e/2}\subset P^\e\subset K^\e;
\ee
there are then some $e_i\in\mathbb{S}^{N-1}$ and $a_i\in\R$ such that~$P^\e$ can then be written as
\be\label{Pe}
P^\e=\bigcap_{i=1}^l\big\{x\in\R^N:\, x\cdot e_i\le a_i\big\};
\ee
\item[(iii)] There exists $m\in\N$, depending on $K,\e,N$, 
such that, for any $z\in\partial K^\e$, there is a compact polyhedron $P^{\e,z}$ with at most $m$ faces such that
\be\label{Pez1}
B_{\e/2}(z-\e \nu)\subset K^{\e/2}\subset P^{\e,z} \subset K^\e\ \hbox{ and }\ z\in P^{\e,z},
\ee
where $\nu$ is the exterior unit normal of $\partial K^\e$ at $z$; the polyhedron $P^{\e,z}$ can then be written as
\be\label{Pez2}
P^{\e,z}=\bigcap_{i=1}^m\big\{x\in\R^N:\, x\cdot f_i\le b_i\big\},
\ee
for some $f_i\in\mathbb{S}^{N-1}$ and $b_i\in\R$ depending on $z$.
\end{itemize}
\end{lemma}

\begin{proof}
(i) Clearly, $K\subset K^\e$, $K^\e \rightarrow K$ as $\e\rightarrow 0$, and it is also readily seen that $K_{\e}$ is convex and compact. Let $z\in\partial K^\e$. Then $d(z,K)=\e$ and there is $y\in \partial K$ such that $|z-y|=\e$. As for \eqref{PbP}, $y$ is the projection of $z$ onto the closed convex set $K$ and
$$K\subset\big\{x\in\R^N:x\.\nu\leq y\.\nu\big\}, \hbox{ where }\nu:=\hat{z-y}=\frac{z-y}{|z-y|}.$$
From this, it follows that 
$$\big\{x\in\R^N:x\.\nu> z\.\nu\big\}\subset \mathbb{R}^N\setminus K^\e.$$
On the other hand, $B_\e(y)\subset K^\e$. Thus, $K^\e$ fulfills both an interior and an exterior ball condition at $z$ (with balls with uniform radius $\e$ independent of $z\in\partial K^\e$). Therefore,~$\partial K^\e$~is of class $C^1$. Furthermore, at $z$, $\nu$ is the outward unit normal to $\partial K^{\e}$.

(ii) Fix $\epsilon>0$. Notice that the conclusions~(ii) and~(iii) hold at once when $N=1$, in which case the sets $K$, $K^{\e/2}$, $K^\e$, $P^\e$ and $P^{\e,z}$ are all non-empty segments, and $l=m=2$ (faces are just points in dimension $N=1$). We can assume in the sequel that $N\ge2$.

The non-empty compact convex set $K^{\e/2}$ can be written as
$$K^{\e/2}=\bigcap_{i\in I}\mc{F}_i,$$
where each $\mc{F}_i$ is a closed affine half-space and $(\partial\mc{F}_i)_{i\in I}$ is the family of all supporting hyperplanes of $K^{\e/2}$. Since $K^{\e/2}\subset{\rm{int}}(K^\e)$ and the non-empty compact set
$$K^{2\e}\cap\overline{\R^N\!\setminus\!K^{\e}}=K^{2\e}\cap(\underbrace{\R^N\!\setminus\!{\rm{int}}(K^\e)}_{\subset\,\R^N\setminus K^{\e/2}})$$
is then covered by the union $\cup_{i\in I}(\R^N\!\setminus\!\mc{F}_i)$ of open sets $\R^N\!\setminus\!\mc{F}_i$, there is a (positive) integer $l=l(K,\e,N)$ depending in general on $(K,\e,N)$ and a family $\{i_1,\cdots,i_l\}\subset I$ such that
$$K^{2\e}\cap\overline{\R^N\!\setminus\!K^\e}\ \subset\ \bigcup_{j=1}^l\,(\R^N\!\setminus\!\mc{F}_{i_j}).$$
Therefore,
$$K^{\e/2}=\bigcap_{i\in I}\mc{F}_i\ \subset\ \underbrace{\bigcap_{j=1}^l\mc{F}_{i_j}}_{=:P^\e}\ \subset\ (\R^N\!\setminus\!K^{2\e})\cup{\rm{int}}(K^\e).$$
The set $P^\e$ is a closed polyhedron of at most $l$ faces, that is,~\eqref{Pe} holds. Since $P^\e$ is convex and the open sets $\R^N\!\setminus\!K^{2\e}$ and ${\rm{int}}(K^\e)$ are disjoint, $P^\e$ is included in one of these two open sets. If $P^\e$ were included in $\R^N\!\setminus\!K^{2\e}$, then so would be $K^{\e/2}$, which is impossible since $K^{\e/2}\subset K^{2\e}$. Therefore, $P^\e\subset{\rm{int}}(K^\e)\subset K^\e$, and~\eqref{KPK} is proved.

(iii) Fix $\epsilon>0$, and consider the closed polyhedron $P^\e$ given in part~(ii). Since $P^\e$ has at most $l=l(K,\e,N)$ faces, there is an integer $k=k(K,\e,N)$ depending on~$l$ and~$N$, and therefore depending on $(K,\e,N)$, such that $P^\e$ has at most~$k$ vertices, whence there are some pairwise distinct points $q_1,\cdots,q_{k'}$ (with $k'\le k$) of $P^\e$ such that~$P^\e$ is equal to the convex hull of $\{q_1,\cdots,q_{k'}\}$.

Consider now any point $z\in\partial K^\e$, and let $P^{\e,z}$ be the convex hull of $\{q_1,\cdots,q_{k'},z\}$, that is, the convex hull of $P^\e\cup\{z\}$. The set $P^{\e,z}$ is a closed polyhedron containing~$z$. Furthermore,~$P^{\e,z}$ has at most $m=m(K,\e,N):=2^{k(K,\e,N)+1}$ faces, that is,~\eqref{Pez2} holds. 

To show~\eqref{Pez1}, observe first that $K^{\e/2}\subset P^\e\subset P^{\e,z}$. Since $P^\e\cup\{z\}\subset K^\e$ and $K^\e$ is convex, one has $P^{\e,z}\subset K^\e$. Lastly, from the proof of~(i) and the notations used there, one knows that $z-\e\nu=y\in\partial K$, where $\nu$ is the exterior unit normal to~$\partial K^\e$ at~$z$. Thus, $B_{\e/2}(z-\e\nu)=B_{\e/2}(y)\subset K^{\e/2}\subset P^{\e,z}$. The proof of Lemma~\ref{lemma:approximation} is thereby complete.
\end{proof}

%%%%%%%%%%%%%%%%%%%%%%%%%%%%%%%%%%%%%%%%%%%%%%%%%%%

\subsection{Some comparison and Liouville-type results}\label{subsection2.3}

In the following Lemmata~\ref{lemma:combu}-\ref{lemma:bistable}, we show some properties of some (planar-like) solutions based on the comparison with pulsating fronts for the ignition or bistable cases, that is, under Hypotheses~\ref{hyp:ignition} or \ref{hyp:bistable}.

\begin{lemma}\label{lemma:combu}
Assume that Hypothesis~\ref{hyp:ignition} holds. For any $\e>0$, there is $\delta_\e \in (0,\sigma)$ such that, for any $\nu\in\mathbb{S}^{N-1}$ and $0<\eta\le \delta_{\e}$, the solution $u(t,x)$ of \eqref{main} with initial datum
$\1_{\{x\cdot\nu\leq0\}}+\eta \1_{\{x\cdot\nu>0\}}$ satisfies
$$\limsup_{t\to+\infty}\Big(\sup_{x\cdot \nu\ge (c^*(\nu) +\e) t} u(t,x)\Big)\le \eta.$$
\end{lemma}

\begin{proof}
For $\eta\in(0,\sigma)$ define
$$f_\eta(x,s):=\max\big(f(x,s),f(x,s+\eta)\big)$$
(recall that $f$ is extended by $0$ outside $\R^N\times[0,1]$). Then $f_\eta$ fulfills Hypothesis~\ref{hyp:ignition} with~$\sigma$ replaced by $\sigma-\eta$.
By \cite{BH1}, in any direction $\nu\in\Sph$, there is a unique (up to shifts in time) pulsating front 
$\Phi_{\eta,\nu}(t,x)=U_{\eta,\nu}(x,x\cdot \nu -c_{\eta}(\nu) t)$ 
for the equation \eqref{main} with $f_{\eta}$ instead of $f$, and its speed $c_\eta(\nu)$ is positive. 
The same proof of \cite[Proposition~2.6]{R1} yields
$$c_{\eta}(\nu)\rightarrow c^*(\nu)\ \ \hbox{as }\;\eta\to 0.$$
Since $c_{\eta}(\nu)$ and $c^*(\nu)$ are continuous with respect to $\nu$ by~\cite[Proposition~2.2]{GHR1}, the above convergence is uniform in $\nu\in\mathbb{S}^{N-1}$.
Then, for $\e>0$, there is $\delta_{\e}\in (0,\sigma)$ such that $c_{\eta}(\nu)<c^*(\nu)+\e$ for any $\eta\in(0,\delta_{\e}]$ and $\nu\in\mathbb{S}^{N-1}$.  
The function $\tilde\Phi(t,x):=\Phi_{\eta,\nu}(t,x)+\eta$ satisfies
$$\partial_t \tilde\Phi-\dv(A(x)\nabla \tilde\Phi)-q(x)\cdot\nabla \tilde\Phi=f_\eta(x,\Phi_{\eta,\nu})
\geq f(x,\tilde\Phi),$$
that is, it is a supersolution for \eqref{main}. In addition, it is larger than $\eta$ and satisfies
$$\inf_{x\.\nu\leq0}\tilde\Phi(t,x)\geq \min_{x\in\R^N,\; z\leq0} 
U_{\eta,\nu}(x,z-c_{\eta}(\nu) t) + \eta\longrightarrow 1+\eta\quad\text{as }\;t\to+\infty.$$
Therefore, we can find $T>0$ so that 
$\tilde\Phi(T,x)\geq \1_{\{x\cdot\nu\leq0\}}+\eta \1_{\{x\cdot\nu>0\}}$ for all $x\in\R^N$. 
The conclusion then follows immediately from the parabolic comparison principle.
\end{proof}

\begin{lemma}\label{lemma:bistable}
Assume that Hypothesis~\ref{hyp:bistable} holds. Let $v:\R\times\R^N\to[0,1]$ be an entire $($defined in $\R\times\R^N$$)$ solution of~\eqref{main}, for which there exist $\nu\in\Sph$ and $c>c^*(\nu)$ such that 
\begin{equation}\label{con:v} 
v(t,x)\le \sigma\ \ \hbox{for all $t\le 0$ and $x\cdot \nu \ge c t$}.
\end{equation}
Then, $v(t,x)\equiv 0$ in $\R\times\R^N$.
\end{lemma}

\begin{proof}
We first show that 
\be\label{lim0}
\lim_{\rho\rightarrow +\infty}\Big(\sup_{\su{t\le0}{x\cdot \nu-ct\ge \rho}}v(t,x)\Big)=0.
\ee
Assume by contradiction that there exist $\delta_1>0$ and sequences $\{t_n\}_{n\in\mathbb{N}}$ of $\R$, $\{x_n\}_{n\in\mathbb{N}}$ of $\R^N$ such that $t_n\le 0$, $v(t_n,x_n)\ge \delta_1$, and $x_n\cdot \nu-ct_n\rightarrow +\infty$ as $n\rightarrow+\infty$. Define $v_{n}(t,x):=v(t+t_n,x+x_n)$ for $(t,x)\in\R\times\R^N$. Then, 
\be\label{vn}
v_n(0,0)\ge \delta_1\hbox{ and } v_n(t,x)\le \sigma \hbox{ for all $t\le -t_n$ and $x\cdot \nu\ge ct-(x_n\cdot \nu -c t_n)$}.
\ee 
Denote $x_n=y_n+z_n$ where $y_n\in [0,1]^N$ and $z_n\in \mathbb{Z}^N$. Then, there is $y_0\in [0,1]^N$ such that $y_n\rightarrow  y_0$ as $n\rightarrow +\infty$ up to extraction of a subsequence. By parabolic estimates, the sequence $\{v_{n}\}_{n\in\N}$ converges as $n\rightarrow +\infty$, up to extraction of a subsequence, locally uniformly to an entire solution $v_{\infty}:\R\times\R^N\to[0,1]$ of a translation of \eqref{main}, that is, 
$$  \partial_t v_{\infty}=\dv(A(x+y_0)\nabla v_{\infty}) +q(x+y_0)\cdot\nabla v_{\infty} + f(x+y_0,v_{\infty}).$$
By \eqref{vn} and $\lim_{n\to+\infty}(x_n\cdot\nu-ct_n)=+\infty$, one also has that $v_{\infty}(0,0)\ge\delta_1$ and $0\le v_{\infty}(t,x)\le \sigma$ for all $t\le 0$ and $x\in\R^N$. Since $\min_{x\in\R^N} f(x,s)<0$ for every $s\in (0,\sigma]$, it then follows easily from the maximum principle, as in~\cite[Lemma~4.2]{GHR1}, that $v_{\infty}(t,x)\equiv 0$, which contradicts $v_{\infty}(0,0)\ge \delta_1$. Thus,~\eqref{lim0} is proved.

Let $\overline{v}(t,x):=\phi_{\nu}(t ,x+c^*(\nu)t\nu)$ and $\underline{v}(t,x):=v(t,x+c^*(\nu)t\nu)$, where $\phi_{\nu}$ is given in Hypothesis~\ref{hyp:bistable}. They satisfy
$$\partial_t u=\dv(A(x+c^*(\nu)t\nu)\nabla u) +c^*(\nu) \nu \cdot\nabla u +q(x+c^*(\nu)t\nu)\cdot \nabla u+ f(x+c^*(\nu)t\nu,u)$$
in $\R\times\R^N$, and all conditions in \cite[Lemma~2.2]{R1} can be verified.\footnote{The functions $\underline{v}$ and $\overline{v}$ are of class $C^{1+\delta/2,2+\delta}(\R\times\R^N)$ for every $0<\delta<1$ by standard parabolic estimates. This regularity is sufficient to apply \cite[Lemma~2.2]{R1} even if in our paper the function $f$ (H\"older continuous in $x$, Lipschitz continuous in $u$) is slightly less regular than the function $f$ (Lipschitz continuous in all variables) used in \cite[Lemma~2.2]{R1}.} Then, it follows from \cite[Lemma~2.2]{R1} that $0\le\underline{v}(t,x)\le \overline{v}(t,x)$ for $(t,x)\in(-\infty,0]\times\R^N$. The comparison principle yields that
$$v(t,x)\le \phi_{\nu}(t,x)=U_{\nu}(x,x\cdot\nu -c^*(\nu) t)\ \hbox{ for all $(t,x)\in \R\times\R^N$}.$$

For any $t_0<0$, there is $y_0=y_0(t_0)\in [-1,1]^N$ such that $[ct_0\nu]:=ct_0\nu-y_0\in \mathbb{Z}^N$ and $y_0\cdot\nu\le0$. 
Since the new function $V:(t,x)\mapsto V(t,x):=v(t+t_0,x+[ct_0 \nu])$ still satisfies \eqref{main} and~\eqref{con:v}, it follows from the arguments of the previous paragraph that $0\le v(t+t_0,x+[ct_0 \nu])\le \phi_{\nu}(t,x)=U_{\nu}(x,x\cdot\nu -c^*(\nu) t)$ for all $(t,x)\in \R\times\R^N$, that is,
$$\baa{rcl}
0\le v(t,x) & \le & U_{\nu}(x-[ct_0 \nu],x\cdot\nu -c^*(\nu) t+(c^*(\nu) -c)t_0+y_0\cdot \nu)\vspace{3pt}\\
& = & U_{\nu}(x,x\cdot\nu -c^*(\nu) t+(c^*(\nu) -c)t_0+y_0\cdot \nu),\eaa$$
for all $(t,x)\in\R\times\R^N$. By taking $t_0\rightarrow -\infty$ and using $c>c^*(\nu)$ and $U_\nu(\cdot,+\infty)=0$, one concludes that $v(t,x)=0$ for all $(t,x)\in\R\times\R^N$.
\end{proof}

%%%%%%%%%%%%%%%%%%%%%%%%%%%%%%%%%%%%%%%%%%%%%%%%%%%
%%%%%%%%%%%%%%%%%%%%%%%%%%%%%%%%%%%%%%%%%%%%%%%%%%%

\section{Lower estimates on the upper-level sets}\label{section3}

This section is devoted to the proofs of Lemma~\ref{lem:subset} and Proposition~\ref{prop:subset} below, which provide some lower estimates for the solutions of~\eqref{main} in some suitable time-dependent sets. The proofs themselves rely on Proposition~\ref{pro:rho} and a superposition principle.

\begin{lemma}\label{lem:subset}
Assume that $U_{\rho}\neq \emptyset$, where $\rho>0$ is given in Proposition~\ref{pro:rho}. For any $0<\tau<1$, it holds
$$\inf_{x\in U_{\rho} +t\tau \W_0} u(t,x)\rightarrow 1\ \hbox{ as $t\rightarrow +\infty$}.$$
\end{lemma}

\begin{proof}
Fix any $\lambda\in(0,1)$, $\tau'\in(\tau,1)$, and $\theta'\in(\theta,1)$, where $\theta\in(0,1)$ is given in Proposition~\ref{pro:rho}. For any $y_0\in\R^N$, let $v_{y_0}$ be the solution of~\eqref{main} with initial datum $v_{y_0}(0,\cdot):=\1_{B_\rho(y_0)}$. By Proposition~\ref{pro:rho}, there is $T_{y_0}>0$ such that $v_{y_0}(T_{y_0},\cdot)\ge\theta'\,\1_{B_{\rho}}$ in~$\R^N$, whence there is $\e_{y_0}>0$ such that $v_y(T_{y_0},\cdot)\ge\theta\,\1_{B_{\rho}}$ in $\R^N$ for all $y\in B_{\e_{y_0}}(y_0)$. By compactness of $[0,1]^N$, there is then a finite number of points $y_1,\cdots,y_k$ in $[0,1]^N$ such that
$$[0,1]^N\subset\bigcup_{i=1}^kB_{\e_{y_i}}(y_i).$$
Call
$$T':=\max(T_{y_1},\cdots,T_{y_k})\in(0,+\infty).$$
Therefore,
\be\label{T'y}
\forall\,y\in[0,1]^N,\ \exists\,T'_y\in(0,T'],\ \ v_y(T'_y,\cdot)\ge\theta\,\1_{B_\rho}\hbox{ in $\R^N$}.
\ee

Let now $w$ be the solution of~\eqref{main} with initial condition $w(0,\cdot):=\theta\,\1_{B_\rho}$. By Proposition~\ref{pro:rho} and~\cite{R1}, $w$ admits $\W_0$ as spreading set. Since $\tau'\W_0\Subset\W_0$ (i.e. $\tau'\W_0$ has compact closure included in the open set $\W_0$), there is then a time $\overline{T}>0$ such that
$$\forall\,s\ge\overline{T},\ \forall\,z\in\W_0,\ \ w(s,s\tau'z)\ge\lambda.$$
Since $\tau\W_0\Subset\tau'\W_0$, there is then a time $T\ge T'$ such that
\be\label{Tw}
\forall\,t\ge T,\, \forall\,s'\in(0,T'],\ \forall\,z\in\W_0,\ \forall\,y\in[0,1]^N,\ \ w(t-s',t\tau z+y)\ge\lambda.
\ee

Finally, consider any $x_0\in U_\rho$. Let $y_0\in[0,1]^N$ be such that $y_0-x_0\in\Z^N$. Since $u_0=\1_U\ge\1_{B_\rho(x_0)}=\1_{B_\rho(y_0)}(\cdot+y_0-x_0)=v_{y_0}(0,\cdot+y_0-x_0)$ and the coefficients of~\eqref{main} are $\Z^N$-periodic, the comparison principle implies that $u(t,x)\ge v_{y_0}(t,x+y_0-x_0)$ for all $(t,x)\in[0,+\infty)\times\R^N$. In particular, by using~\eqref{T'y}-\eqref{Tw} and the comparison principle again, one gets that
$$\forall\,t\ge T,\ \forall\,z\in\W_0,\ \ u(t,x_0+t\tau z)\ge v_{y_0}(t,t\tau z+y_0)\ge w(t-T'_{y_0},t\tau z+y_0)\ge\lambda.$$
Since $\lambda\in(0,1)$ was arbitrary and $u\le1$ in $[0,+\infty)\times\R^N$, the proof is~complete.
\end{proof}

The previous lemma, combined with compactness and limit arguments, implies that the set $\mc{C}(\mc{U}(U_{\rho}))+\W_0$ is a spreading subset for $u$, in the sense given in the last sentence of Definition~\ref{def:W}.

\begin{proposition}\label{prop:subset}
Assume one of Hypotheses~\ref{hyp:KPP}-\ref{hyp:bistable}, and that $U_{\rho}\neq \emptyset$, with $\rho>0$ given in Proposition~\ref{pro:rho}. Then, the set $\mc{C}(\mc{U}(U_{\rho}))+\W_0$ is a spreading subset for $u$. 
\end{proposition}

\begin{proof}
When $\mc{U}(U_\rho)=\emptyset$, then $\mc{C}(\mc{U}(U_{\rho}))+\W_0=\{0\}+\W_0=\W_0$ is a spreading subset of $u$, since $1\ge u_0\ge\1_{B_\rho(x_0)}$ in $\R^N$ for any $x_0\in U_\rho$ and $\W_0$ is the spreading set of the solution of~\eqref{main} with initial condition~$\1_{B_\rho(x_0)}$ (by Proposition~\ref{pro:rho} and~\cite{R1}).

One can assume in the sequel to $\mc{U}(U_\rho)\neq\emptyset$. For any $\tau>0$ and $\xi\in \mathcal{U}(U_{\rho})$, it follows from the definition of $\mc{U}(U_\rho)$ that
$$d(t\tau\xi,U_{\rho})=o(t)\ \hbox{as }t\rightarrow +\infty.$$
This implies that, for any $\eta\in (0,1)$ and any $\vartheta\in(\eta,1)$, there holds $t\tau\xi + t\eta \W_0\subset U_{\rho} +t\vartheta\W_0$ for~$t$ sufficiently large, since $\eta\W_0\Subset\vartheta\W_0$. By Lemma~\ref{lem:subset}, one then gets that
\be\label{taueta}
\inf_{x\in \tau\xi +\eta \W_0}u(t,tx)\rightarrow 1\ \hbox{ as }t\rightarrow +\infty.
\ee

Let $\W$ be the open set defined by $\W:=\mc{C}(\mc{U}(U_{\rho}))+\W_0$. For any $x\in\W$, one has $(1+\e)x\in\W$ for all $\e>0$ small enough, that is, $x\in(1+\e)^{-1}\W=\mc{C}(\mc{U}(U_{\rho}))+(1+\e)^{-1}\W_0$ for all $\e>0$ small enough, whence there are $\tau_x\ge 0$, $\xi_x\in\mathcal{U}(U_{\rho})$ and $\eta_x\in (0,1)$ such that $x\in\tau_x\xi_x+\eta_x\W_0$. By openness of $\eta_x\W_0$, one can assume without loss of generality that $\tau_x>0$.

Consider finally any non-empty compact set $C\subset \W$. It can then be covered by a finite number of  open sets $\tau\xi+\eta \W_0$ with $\tau>0$, $\xi\in\mathcal{U}(U_{\rho})$ and $\eta\in (0,1)$. It then follows from~\eqref{taueta} that
$$\lim_{t\rightarrow +\infty} \Big(\min_{x\in C} u(t,tx)\Big)=1.$$
This means that $\W=\mc{C}(\mc{U}(U_{\rho}))+\W_0$ is a spreading subset, which is the desired conclusion.
\end{proof}

%%%%%%%%%%%%%%%%%%%%%%%%%%%%%%%%%%%%%%%%%%%%%%%%%%%
%%%%%%%%%%%%%%%%%%%%%%%%%%%%%%%%%%%%%%%%%%%%%%%%%%%

\section{Upper estimates on the upper-level sets}\label{section4}

This section is the core of the paper. We prove in Proposition~\ref{prop:superset} below in Section~\ref{subsection4.2} some upper bounds of the upper-level sets of the solutions of~\eqref{main}, under assumption~\eqref{supt-plane}. The proof relies on the derivation, in Section~\ref{subsection4.1}, of some quantitative upper estimates for auxiliary and so-called retracting solutions which are initially supported on the complements of polyhedral compact sets approximating large dilations of the set~$-\W_0$.

%%%%%%%%%%%%%%%%%%%%%%%%%%%%%%%%%%%%%%%%%%%%%%%%%%%

\subsection{Retracting solutions}\label{subsection4.1}

Under one of the Hypotheses~\ref{hyp:KPP}-\ref{hyp:bistable}, let $\W_0$ be given by~\eqref{w0}-\eqref{W0-wulff}. The set $\overline{\W_0}$ is a non-empty compact convex set. If~\eqref{supt-plane} further holds, then the exterior $C^1$ and polyhedral approximations of~$\overline{\W_0}$ given in Lemma~\ref{lemma:approximation} have the following properties.

\begin{lemma}\label{lemma:app-W0}
Assume one of Hypotheses~\ref{hyp:KPP}-\ref{hyp:bistable}, let $\W_0$ be given by~\eqref{w0}-\eqref{W0-wulff}, and assume that~\eqref{supt-plane} holds. Set $K:=\overline{\W_0}$ and, for $\e>0$, let $K^\e:=\overline{\W_0}+\overline{B_\e}$, $P^\e$, and $P^{\e,z}$ be the exterior $C^1$ and polyhedral approximations of $\overline{\W_0}$ given in Lemma~\ref{lemma:approximation}, $P^{\e,z}$ being related to an arbitrary $z\in \partial K^\e$. Then, it holds:
\begin{itemize}
\item[(i)] $z\cdot \nu=c^*(\nu)+\e$, where $\nu$ is the exterior unit normal to~$\partial K^\e$ at~$z$;
\item[(ii)] $c^*(e_i)\le a_i$ for all $1\le i\le l$, resp. $c^*(f_i)\le b_i$ for all $1\le i\le m$, where~$e_i$,~$a_i$, and~$l$ are given in~\eqref{Pe} for $P^\e$, resp.~$f_i$,~$b_i$, and~$m$ are given in~\eqref{Pez2} for $P^{\e,z}$.
\end{itemize}
\end{lemma}

\begin{proof}
(i) It follows from the proof of (i) of Lemma~\ref{lemma:approximation}  that $z-\e\nu\in\partial\overline{\W_0}=\partial\W_0$, and
$$\overline{\W_0}\subset\big\{x\in\R^N:x\.\nu\leq (z-\e\nu)\cdot \nu=z\.\nu-\e\big\}.$$
On the one hand, the fact that $z-\e\nu\in\partial\W_0$ immediately yields by~\eqref{W0-wulff} that $z\.\nu-\e\leq c^*(\nu)$. On the other hand,  by \eqref{supt-plane}, there exists $\bar x\in\partial \W_0$ such that $\bar x\. \nu=c^*(\nu)$. As a consequence, one gets that $c^*(\nu)=\bar x\.\nu\leq z\.\nu-\e$. Thus, $z\.\nu= c^*(\nu)+\e$.

(ii) Since $K=\overline{\W_0}\subset\overline{\W_0}+\overline{B_{\e/2}}=K^{\e/2}\subset P^\e$ by~\eqref{KPK}, one has that
$$x\cdot e_i\le a_i\ \hbox{ for all $1\le i\le l$ and $x\in\overline{\W_0}$}.$$
By \eqref{supt-plane}, for each $1\le i\le l$, there is $\bar{x}_i\in \partial\W_0$ such that $\bar{x}_i \cdot e_i =c^*(e_i)$ whence $c^*(e_i)=\bar{x}_i\cdot e_i\le a_i$. The conclusion is similar for $P^{\e,z}$, with $c^*(f_i)\le b_i$ for all $1\le i\le m$. 
\end{proof}

We now introduce the {\em retracting solutions}. For given $\Omega\subset\R^N$ and $\kappa>0$, we let $v_{-\kappa\Omega}$ be the solution to \eqref{main} with initial datum
$$v_{-\kappa\Omega}(0,\cdot)=\1_{\R^N\setminus(-\kappa\Omega)}.$$
Roughly speaking, under one of the Hypotheses~\ref{hyp:KPP}-\ref{hyp:bistable}, we show in the following Lemmata~\ref{extract-kpp}-\ref{lem:vk<e} that the region where $v_{-\kappa\W_0}$ is close to $0$ retracts as the shape of $-\W_0$ for $\kappa>0$ large enough, where $\W_0$ is given by~\eqref{w0}-\eqref{W0-wulff}. To do so, we use Lemmata~\ref{lemma:approximation} and~\ref{lemma:app-W0} on the exterior approximations of $\overline{\W_0}$. We first consider the Fisher-KPP case, that is, Hypothesis~\ref{hyp:KPP}.

\begin{lemma}\label{extract-kpp}
Assume that Hypothesis~\ref{hyp:KPP} and \eqref{supt-plane} hold, 
with $\W_0$ given in~\eqref{w0}-\eqref{W0-wulff}. Let $\e>0$ and $\eta>0$ be arbitrary, and let $P^\e$ be a polyhedral exterior approximation of $K:=\overline{\W_0}$ given in Lemma~\ref{lemma:approximation}-(ii). Then there exists $R>0$ such that, for all $\kappa\geq2R$,
\begin{equation}\label{kW}
v_{-\kappa P^\e}(t,x)\leq\eta\ \hbox{ for all } t\in[0,\kappa-2R]\hbox{ and }x\in (-\kappa+R+t)P^\e.
\end{equation}
\end{lemma}

\begin{proof}
For $e\in\mathbb{S}^{N-1}$ and $\lambda\in \R$, call $k(e,\lambda)$ the principal eigenvalue of the operator
\begin{equation}\label{L-FK}
L_{e,\lambda} \varphi:=-\nabla\cdot (A\nabla\varphi) +2\lambda eA\nabla \varphi -q\cdot\nabla \varphi +[\lambda \nabla\cdot (A e) +\lambda q\cdot e -\lambda^2 eAe -f_s(x,0)]\varphi,
\end{equation}
acting on the set
\be\label{Eelambda}
E:=\{\varphi\in C^2(\R^N):\varphi \hbox{ is $\Z^N$-periodic}\}.
\ee
It follows from \cite[Remark 1.16]{BH1} and \cite[Section~6.4]{BH1} that, for any pulsating traveling front connecting $1$ to $0$, and propagating in the direction $e$ with speed $c\ge c^*(e)$, there exists $\lambda>0$ such that $k(e,\lambda)+c\lambda=0$.

Recall that, from Lemma~\ref{lemma:approximation}-(ii), there exist $l=l(\overline{\W_0},\e,N)\in\N$, some directions $(e_i)_{1\le i\le l}$ of $\Sph$, and some real numbers $(a_i)_{1\le i\le l}$ such that
\be\label{-Pe}
-P^\e=\bigcap_{i=1}^l\big\{x\in\R^N:\, -x\cdot e_i\le a_i\big\},
\ee
where $a_i\ge c^*(e_i)>0$ by Lemma~\ref{lemma:app-W0}-(ii). Then, for each $1\le i\le l$, let $\lambda_i>0$ be a root of $k(e_i,\lambda_i)+a_i\lambda_i=0$ and let $\varphi_i$ be the related normalized principal eigenfunction of~$L_{e_i,\lambda_i}$, that is, satisfying
$$L_{e_i,\lambda_i} \varphi_i=k(e_i,\lambda_i) \varphi_i\ \hbox{ in $\R^N$}$$
with $\varphi_i\in E$, $\|\varphi_i\|_{L^{\infty}(\R^N)}=1$, and $\varphi_i>0$ in $\R^N$. Then, by Hypothesis~\ref{hyp:KPP}, for every $1\le i\le l$ and every $\varrho\in\R$, the function $u(t,x):=e^{-\lambda_i(x\cdot e_i-a_it-\varrho)}  \varphi_i(x)$ satisfies
\begin{equation}\label{linEq}\baa{l}
\partial_t u-\hbox{div}(A(x)\nabla u)-q(x)\cdot \nabla u -f(x,u)\vspace{3pt}\\
\qquad\qquad\qquad\qquad\ge \partial_t u-\hbox{div}(A(x)\nabla u)-q(x)\cdot \nabla u-f_s(x,0)u=0\eaa
\end{equation}
in $\R\times\R^N$. Let $\tau>0$ be large enough such that
\begin{equation}\label{Re}
e^{\lambda_i\tau a_i}  \varphi_i(x)\ge 1\ \hbox{ and }\ e^{-\lambda_i \tau a_i}  \varphi_i(x)\le \frac{\eta}{l}\ \hbox{ for all $i=1,\cdots,l$, and for all $x\in\R^N$}.
\end{equation}

Define, for $\kappa>0$, $t\ge0$ and $x\in\R^N$,
$$u^+(t,x):=\min\left\{ \sum_{i=1}^{l} e^{-\lambda_i(x\cdot e_i+a_i(\kappa-\tau-t))} \varphi_i(x),1\right\}.$$
By \eqref{linEq} and Hypothesis~\ref{hyp:KPP}, one can easily check that 
$$\partial_t u^+-\hbox{div}(A(x)\nabla u^+)-q(x)\cdot \nabla u^+-f(x,u^+)\ge 0,$$
for all $t\ge 0$ and $x\in\R^N$ such that $u^+(t,x)<1$. Furthermore, for every $x\in \R^N\!\setminus\!(-\kappa P^\e)$, there is $i\in \{1,\cdots,l\}$ such that $-x\cdot e_i>\kappa a_i$, which implies by~\eqref{Re} that
$$e^{-\lambda_i(x\cdot e_i+a_i(\kappa-\tau))} \varphi_i(x)\ge e^{\lambda_i\tau a_i} \varphi_i(x)\ge 1.$$
Hence, $u^+(0,x)=1=v_{-\kappa P^\e}(0,x)$ for $x\in \R^N\!\setminus\!(-\kappa P^\e)$, while $u^+(0,x)>0=v_{-\kappa P^\e}(0,x)$ for $x\in-\kappa P^\e$. Thus, $u^+(0,\cdot)\ge v_{-\kappa P^\e}(0,\cdot)$ in $\R^N$, and the comparison principle yields
$$0\le v_{-\kappa P^\e}(t,x)\le u^+(t,x)\ \hbox{ for all $t\ge 0$ and $x\in\R^N$}.$$
Especially, by taking $R:=2\tau$, and then any $\kappa\ge 2R$, one has by~\eqref{-Pe} and~\eqref{Re} that 
$$0\le v_{-\kappa P^\e}(t,x)\le u^+(t,x)\le \sum_{i=1}^{l} e^{-\lambda_i \tau a_i} \varphi_i(x)\le \eta, $$
for all $t\in[0,\kappa-2R]$ and $x\in (-\kappa+R+t)P^\e=-|-\kappa+R+t|P^\e$.
\end{proof}

We then consider the ignition or weakly bistable cases, that is, Hypotheses~\ref{hyp:ignition} or~\ref{hyp:bistable}. We show similar estimates as in Lemma~\ref{extract-kpp}, by arguing by contradiction and using, this time, the $C^1$ exterior approximations of $\overline{\mc{W}_0}$ together with Lemmata~\ref{lemma:combu}-\ref{lemma:bistable}.

\begin{lemma}\label{lem:vk<e}
Assume that  one of Hypotheses~\ref{hyp:ignition}-\ref{hyp:bistable} holds, as well as~\eqref{supt-plane}, with $\W_0$ given in~\eqref{w0}-\eqref{W0-wulff}. Let $\e>0$ and $\eta>0$ be arbitrary, and let $K^\e:=\overline{\W_0}+\overline{B_\e}$ be the $C^1$ exterior approximation of $K:=\overline{\W_0}$ given in Lemma~\ref{lemma:approximation}-(i). Then there exists $R>0$ such that~\eqref{kW} holds for all $\kappa\geq2R$ with $K^\e$ instead of $P^\e$.
\end{lemma}

\begin{proof}
We argue by contradiction, supposing that the statement fails for some $\e>0$ and~$\eta>0$. We can then assume without loss of generality that $0<\e\le1$ and $0<\eta\le \delta_{\e}\le \sigma$, where $\sigma$ is given in Hypotheses~\ref{hyp:ignition} or~\ref{hyp:bistable}, and $\delta_\e$ is given by Lemma~\ref{lemma:combu} (in case of Hypothesis~\ref{hyp:ignition}). For any $\kappa>0$, denote the function $v_{-\kappa K^\e}$ by $v_{\kappa}$ for short.

Then, from our assumption, there exist some sequences $\seq{\kappa}$ in $(0,+\infty)$, $\seq{t}$ in $\R$ and~$\seq{x}$ in $\R^N$ satisfying:
\Fi{tnxn}
\kappa_n\geq 2n,\ \ 0\le t_n\leq \kappa_n-2n,\ \ x_n\in (-\kappa_n+n+t_n)K^\e,\ \ \eta<v_{\kappa_n}(t_n,x_n)\le1,
\Ff
for all $n\in\N$. We have in particular
$$-\kappa_n+n\leq -\kappa_n+n+t_n\leq -n\le -1$$
for all $n\ge 1$, and
\Fi{<xn<}
x_n\in (-\kappa_n+n)K^\e,
\Ff
because $K^\e$ is in particular star-shaped with respect to the origin.
	
Since $0\in\W_0$, one has $B_\e\subset K^\e$, whence $(-\kappa_n+n)K^\e\supset B_{(\kappa_n-n)\e}$. Applying Lemma~\ref{lem:distW} to the set $(-\kappa_n+n)K^\e$ (instead of $K$), we deduce that, for each $n\ge1$,
\Fi{k+n}
\min_{x\in (-\kappa_n+n)K^\e}\ d\big(x,\R^N\!\setminus\!(-\kappa_nK^\e)\big)\geq \Big(\frac{-\kappa_n}{-\kappa_n+n}-1\Big)(\kappa_n-n)\e=n\e>0.
\Ff
Since $v_{\kappa_n}(0,\.)=0$ in ${\rm{int}}(-\kappa_nK^\e)\supset(-\kappa_n+n)K^\e$, the function $v_{\kappa_n}$ is continuous in~$[0,+\infty)\times(-\kappa_n+n)K^\e$ and thus there exist a time and a point, that we still call~$t_n$ and~$x_n$ without loss of generality, fulfilling~\eqref{tnxn} together with $t_n>0$ and
\Fi{tangency}
\forall\,t\in[0,t_n), \ \ \forall\,x\in (-\kappa_n+n+t)K^\e,\ \ v_{\kappa_n}(t,x)<\eta\ \hbox{ and }\ v_{\kappa_n}(t_n,x_n)=\eta.
\Ff

We now claim the following:
\Fi{xnDW}
\forall\,n\in\N,\ \ x_n\in(-\kappa_n+n+t_n)\partial K^\e,
\Ff
and
\Fi{tninfty}
t_n\to+\infty\ \hbox{ as }n\to+\infty.
\Ff	
Property \eqref{xnDW} immediately follows from the strong parabolic maximum principle, because the constant $\eta$ is a supersolution to~\eqref{main} and $v_{\kappa_n}(0,\.)=0<\eta$ in $(-\kappa_n+n)K^\e$. Property \eqref{tninfty} follows instead from parabolic estimates. Indeed,~\eqref{<xn<} and \eqref{k+n} entail that $d(x_n,\R^N\!\setminus\!(-\kappa_nK^\e))\to+\infty$ as $n\to+\infty$ and thus, since $v_{\kappa_n}(0,\.)=0$ in~$-\kappa_nK^\e$, the parabolic estimates imply that the functions $(t,x)\mapsto v_{\kappa_n}(t,x+x_n)$ converge locally uniformly in $[0,+\infty)\times\R^N$, up to a subsequence, to a solution $\t v$ of some translation of the equation~\eqref{main} with initial datum $\t v(0,\cdot)=0$ in $\R^N$, that is, $\t v\equiv0$ in $[0,+\infty)\times\R^N$. However, since $v_{\kappa_n}(t_n,x_n)=\eta$, this shows that $\seq{t}$ cannot have a bounded subsequence, i.e.,~\eqref{tninfty} holds too.

Let us then define, for $(t,x)\in[-t_n,+\infty)\times\R^N$, 
$$ v^n(t,x):=v_{\kappa_n}(t+t_n,x+x_n).$$ 
Then, \eqref{tangency} can be rewritten as
\Fi{vn<e}
\forall\,t\in[-t_n,0),\ \ \forall\,x\in (-\kappa_n+n+t_n+t)K^\e-\{x_n\},\ \ v^n(t,x)<\eta\ \hbox{ and }\ v^n(0,0)=\eta.
\Ff
Call
\be\label{defbarxn}
\bar x_n:=\frac{x_n}{-\kappa_n+n+t_n},
\ee
which belongs to $\partial K^\e$ owing to~\eqref{xnDW}. Up to extraction of a subsequence, the sequence $\seq{\bar x}$ converges to some $\bar x\in\partial K^\e$. Let $\nu_n$ and $\nu$ denote the exterior normals to $K^\e$ at $\bar x_n$ and $\bar x$ respectively. The regularity of $\partial K^\e$ implies that $\seq{\nu}$ converges to $\nu$ and, roughly speaking, the dilations of $K^\e$ will converge to half-spaces locally uniformly around boundary points, as we see below.

More precisely, we introduce the family of open half-spaces $(H_t)_{t<0}$ defined by
$$H_t:=\big\{x\in\R^N\ :\ x\.\nu>(c^*(\nu)+\e)t\big\}.$$
Fix in this paragraph $t<0$ and $x\in H_t$. We claim that
\be\label{vneta}
v^n(t,x)<\eta\ \hbox{ for all $n$ sufficiently large}.
\ee
For all $n$ large enough, we have $t\in[-t_n,0)$, while $-\kappa_n+n+t_n+t<-\kappa_n+n+t_n\le -n<0$. Thus, owing to~\eqref{vn<e}, in order to have~\eqref{vneta}, it is sufficient to show that $x+x_n\in (-\kappa_n+n+t_n+t)K^\e$, i.e.
$$z_n:=\frac{x+x_n}{-\kappa_n+n+t_n+t}\in K^\e.$$
Using that $x_n=(-\kappa_n+n+t_n)\bar x_n$  we find
$$|z_n-\bar x_n|=\frac{|x-t\bar x_n|}{\kappa_n-n-t_n-t},$$
which converges to $0$ as $n\to+\infty$ because $\seq{\bar x}$ is bounded while $\kappa_n-n-t_n-t\ge n-t$, by~\eqref{tnxn}. Next, using $\bar x\.\nu=c^*(\nu)+\e$ from Lemma~\ref{lemma:app-W0}-(i), we have $t\overline{x}\cdot \nu=(c^*(\nu)+\e)t<x\cdot \nu$, whence $t\overline{x}\neq x$. Therefore, $|x-t\bar x_n|>0$ and $|z_n-\bar x_n|>0$ for all $n$ large enough and, using $\bar x_n\.\nu_n=c^*(\nu_n)+\e$ from Lemma~\ref{lemma:app-W0}-(i) again,		
$$\frac{z_n-\bar x_n}{|z_n-\bar x_n|}\.\nu_n=-\frac{x\.\nu_n-t(c^*(\nu_n)+\e)}{|x-t\bar x_n|}\to-\frac{x\.\nu-t(c^*(\nu)+\e)}{|x-t\bar x|}\ \hbox{ as } n\to+\infty.$$
We notice that the above limit is strictly negative because $x\in H_t$. By the $C^1$ regularity of $K^\e$ and the fact that $|z_n-\bar x_n|\to0$ as $n\rightarrow +\infty$, we deduce that $z_n\in K^\e$ for all $n$ sufficiently large, whence $v^n(t,x)<\eta$ thanks to~\eqref{vn<e}. Therefore,~\eqref{vneta} is proved.

Now, the sequence $(v^n)_{n\in\N}$ converges locally uniformly in $\R\times\R^N$, up to a subsequence, to an entire solution $v^*:\R\times\R^N\to[0,1]$ of some translation of the equation~\eqref{main}, satisfying the following properties:
\begin{equation}\label{vstar}
v^*(0,0)=\eta,\ \hbox{ and }\ v^*(t,x)\leq\eta\hbox{ for all }t\leq0\hbox{ and }x\in H_t.
 \end{equation}

{\it Case 1: Hypothesis~\ref{hyp:bistable} holds}. Here, one gets a contradiction from \eqref{vstar} by applying Lemma~\ref{lemma:bistable} to (a time-space translation of) $v^*$.

{\it Case 2: Hypothesis~\ref{hyp:ignition} holds}. Here, one gets from \eqref{vstar} and Lemma~\ref{lemma:combu} that $v^*(t,x)\le \eta$ for all $(t,x)\in\R\times\R^N$. Then, from the strong maximum principle,
\be\label{v*eta}
v^*(t,x)=\eta\ \hbox{ for all $(t,x)\in\R\times\R^N$}.
\ee
To get a contradiction, we need further arguments.

For each $e\in\mathbb{S}^{N-1}$ and $\lambda\in\R$, call $k(e,\lambda)$ the principal eigenvalue of the operator
$$\mathscr{L}_{e,\lambda} \varphi:=-\nabla\cdot (A\nabla\varphi) +2\lambda eA\nabla \varphi -q\cdot \nabla\varphi  +[\lambda \nabla\cdot (A e) +\lambda q\cdot e -\lambda^2 eAe]\varphi,$$
which is actually the same as the operator $L_{e,\lambda}$ in~\eqref{L-FK} since here $f_s(x,0)\equiv0$, acting on the same set~$E$ as in~\eqref{Eelambda}. By evaluating the equality $\mathscr{L}_{e,\lambda} \varphi=k(e,\lambda)\varphi$ at a minimum point $x\in[0,1]^N$ of a principal (that is, positive) eigenfunction $\varphi$, one gets that $k(e,\lambda)\le\lambda\nabla\cdot(Ae)(x)+\lambda q(x)\cdot e -\lambda^2 eA(x)e$, whence
\be\label{kelambda1}
\limsup_{\lambda\rightarrow+\infty}\Big(\sup_{e\in\mathbb{S}^{N-1}}\frac{k(e,\lambda)}{\lambda^2}\Big)<0.
\ee
Notice that $k(e,0)=0$ for all $e\in\Sph$, with any positive constant function as principal eigenfunction. Furthermore, we claim that
\be\label{kelambda2}
\sup_{e\in\mathbb{S}^{N-1}}|k(e,\lambda)|=o(\lambda)\ \hbox{ as }\lambda\to0.
\ee
To see it, consider any sequence $(e_n)_{n\in\N}$ in $\Sph$ and $(\lambda_n)_{n\in\N}$ in $\R^*$ such that $\lambda_n\to0$ as $n\to+\infty$, and let us show that $k(e_n,\lambda_n)/\lambda_n\to0$ as $n\to+\infty$. By evaluating the equality $\mathscr{L}_{e_n,\lambda_n} \varphi_n=k(e_n,\lambda_n)\varphi_n$ at a minimum and a maximum point of a principal (that is, positive) eigenfunction $\varphi_n$, one gets that
\be\label{ken}
|k(e_n,\lambda_n)|\le|\lambda_n|\,\|\nabla\cdot(Ae_n)\|_{L^\infty(\R^N)}+|\lambda_n|\|q\|_{L^\infty(\R^N)}+\lambda_n^2\|e_nAe_n\|_{L^\infty(\R^N)},
\ee
whence $k(e_n,\lambda_n)\to0$ as $n\to+\infty$. By normalizing the principal positive $\Z^N$-periodic eigenfunction $\varphi_n$ with $\max_{\R^N}\varphi_n=1$, it follows from standard elliptic estimates that, up to extraction of a subsequence, the sequence $(\varphi_n)_{n\in\N}$ converges in $C^2(\R^N)$ to a non-negative $\Z^N$-periodic solution $\varphi_\infty$ of $-\nabla\cdot (A\nabla\varphi_\infty)-q\cdot \nabla\varphi_\infty=0$ in $\R^N$, with $\max_{\R^N}\varphi_\infty=1$. Thus, $\varphi_\infty=1$ in $\R^N$, from the strong maximum principle. Next, by integrating the equality $\mathscr{L}_{e_n,\lambda_n} \varphi_n=k(e_n,\lambda_n)\varphi_n$ over $(0,1)^N$ and using the $\Z^N$-periodicity of the coefficients $A$ and $q$, together with~\eqref{hypq}, one infers that
$$\frac{k(e_n,\lambda_n)}{\lambda_n}\int_{(0,1)^N}\varphi_n=\int_{(0,1)^N}e_nA\nabla\varphi_n+\int_{(0,1)^N}(q\cdot e_n)\varphi_n-\int_{(0,1)^N}\lambda_ne_nAe_n\varphi_n.$$
Using again~\eqref{hypq} and the convergence of $(\varphi_n)_{n\in\N}$ to $1$ in $C^2(\R^N)$, one concludes that $k(e_n,\lambda_n)/\lambda_n\to0$ as $n\to+\infty$. As a consequence,~\eqref{kelambda2} is proved.

Now, it follows from \cite[Proposition~5.7]{BH1} that, for each $e\in\Sph$, the function $\lambda\mapsto k(e,\lambda)$ is concave in $\R$. Together with~\eqref{kelambda1}-\eqref{kelambda2}, one infers that, for each $e\in\Sph$ and $\gamma>0$, there exists a unique $\lambda_{e,\gamma}>0$ such that
$$k(e,\lambda_{e,\gamma})+\gamma\lambda_{e,\gamma}=0.$$
Call $\varphi_{e,\gamma}$ a principal eigenfunction of the operator $\mathscr{L}_{e,\lambda_{e,\gamma}}$. Let $\varrho>0$ be such that $c^*(e)\le\varrho$ for all $e\in\Sph$, and remember also that the map $e\mapsto c^*(e)$ is continuous and positive in $\Sph$. One then infers from~\eqref{kelambda1}-\eqref{kelambda2} that
\be\label{lowerlambda}
0<\inf_{\su{e\in\Sph}{\gamma\in[c^*(e),\varrho+1]}}\lambda_{e,\gamma}\le\sup_{\su{e\in\Sph}{\gamma\in[c^*(e),\varrho+1]}}\lambda_{e,\gamma}<+\infty,
\ee
and then, from upper bounds similar to those in~\eqref{ken}, one gets that
$$\sup_{\su{e\in\Sph}{\gamma\in[c^*(e),\varrho+1]}}|k(e,\lambda_{e,\gamma})|<+\infty.$$
From the Harnack inequality applied to the $\Z^N$-periodic positive $C^2(\R^N)$ functions $\varphi_{e,\gamma}$, there is finally a positive real number $M$ such that
\be\label{harnack}
\sup_{\su{e\in\Sph}{\gamma\in[c^*(e),\varrho+1]}}\frac{\max_{\R^N}\varphi_{e,\gamma}}{\min_{\R^N}\varphi_{e,\gamma}}\le M.
\ee
Even if it means decreasing $\eta>0$, one can assume without loss of generality that
$$0<\eta\le\min\Big(1,\delta_{\e},\frac{\sigma}{Mm}\Big),$$
where we recall that $\sigma\in(0,1)$ is given in Hypothesis~\ref{hyp:ignition}, $\delta_\e\in(0,\sigma)$ is given in Lemma~\ref{lemma:combu}, and $m=m(\overline{\W_0},\e,N)\in\N$ is given in Lemma~\ref{lemma:approximation}-(iii) with $K=\overline{\W_0}$.

Coming back to the sequence $(t_n,x_n)_{n\in\N}$ in $(0,+\infty)\times\R^N$ satisfying~\eqref{tangency}-\eqref{tninfty} and to the points $\bar{x}_n\in\partial K^\e$ defined in~\eqref{defbarxn}, let now $P^{\e,\bar{x}_n}$ be a closed polyhedral exterior approximation of $K=\overline{\W_0}$ given in Lemma~\ref{lemma:approximation}-(iii) and used in Lemma~\ref{lemma:app-W0}-(ii), such that $\bar{x}_n\in P^{\e,\bar{x}_n}\subset K^\e$ and
$$-P^{\e,\bar{x}_n}=\bigcap_{i=1}^m\big\{x\in\R^N:-x\cdot f^n_i\le b^n_i\big\},$$
where $f^n_i\in\Sph$ and $b^n_i\ge c^*(f^n_i)>0$ for every $1\le i\le m$. Recall that $m$ is independent of $n$, but the directions $f^n_i$ and the real numbers $b^n_i$ depend on $n$ in general. Furthermore, since $P^{\e,\bar{x}_n}\subset K^\e=\overline{\W_0}+\overline{B_\e}\subset\overline{B_{\varrho+1}}$ (remember that $0<\omega_0(e)\le c^*(e)\le\varrho$ for every $e\in\Sph$ and that $0<\epsilon\le1$), one can assume without loss of generality that $b^n_i\le\varrho+1$ for all $n\in\N$ and $1\le i\le m$. Then, for each $n\in\N$ and $1\le i\le m$, let $\lambda_i^n:=\lambda_{f^n_i,b^n_i}$ be the positive root of
$$k(f^n_i,\lambda^n_i)+b^n_i\lambda^n_i=0$$
and let $\varphi_i^n:=\varphi_{f^n_i,b^n_i}\in E$ be the related normalized principal eigenfunction, that is, satis\-fying $\mathscr{L}_{f^n_i,\lambda^n_i}\varphi^n_i=k(f^n_i,\lambda^n_i) \varphi^n_i$ in $\R^N$, $\varphi_i^n>0$ in $\R^N$ and $\|\varphi^n_i\|_{L^\infty(\R^N)}=1$. From~\eqref{harnack}, the functions $\varphi^n_i$ also satisfy
$$\min_{\R^N}\varphi^n_i\ge\frac{\displaystyle\max_{\R^N}\varphi^n_i}{M}=\frac{1}{M},$$
with $M>0$ independent of $n$ and $i$.
 
Define, for $n\in\N$ and $(t,x)\in \R\times\R^N$,
$$u^+_n(t,x):= \sum_{i=1}^{m} \frac{\eta }{\displaystyle\min_{\R^N} \varphi_i^n}e^{-\lambda_i^n(x\cdot f^n_i +b^n_i(\kappa_n-n-t))} \varphi^n_i(x).$$
Then
$$0<u^+_n(t,x)\le \eta M m\le \sigma\ \hbox{ for all $t\in [0,t_n]$ and $x\in (-\kappa_n+n+t)P^{\e,\bar{x}_n}$},$$
and
$$u^+_n(t,x)\ge \eta\ \hbox{ for all $t\in [0,t_n]$ and $x\in (-\kappa_n+n+t)\partial P^{\e,\bar{x}_n}$}.$$
Owing to the definition of $\lambda^n_i$ and $\varphi^n_i$, one can easily check that 
\begin{equation}\label{eq:u+}
\partial_t u^+_n-\hbox{div}(A(x)\nabla u^+_n)-q(x)\cdot \nabla u^+_n= 0\ \hbox{ in }[0,t_n]\times(-\kappa_n+n+t)P^{\e,\bar{x}_n}
\end{equation}
(and even in $\R\times\R^N$).

Regarding $v_{\kappa_n}$, we have that $v_{\kappa_n}(0,x)=0\le u^+_n(0,x)$ for $x\in (-\kappa_n +n)P^{\e,\bar{x}_n}\subset(-\kappa_n+n)K^\e$ and
$$v_{\kappa_n}(t,x)\le\eta\le u^+_n(t,x)\ \hbox{ for all $t\in [0,t_n]$ and $x\in (-\kappa_n+n+t)\partial P^{\e,\bar{x}_n}$},$$ 
by \eqref{tangency} and $P^{\e,\bar{x}_n}\subset K^\e$. Since $f(x,s)=0$ for $(x,s)\in\R^N\times[0,\sigma]$, it follows from \eqref{tangency} and $P^{\e,\bar{x}_n}\subset K^\e$ that  $v_{\kappa_n}$ also satisfies \eqref{eq:u+} in $(0,t_n]\times(-\kappa_n+n+t)P^{\e,\bar{x}_n}$. Therefore, by the comparison principle,
$$v_{\kappa_n}(t,x)\le u^+_n(t,x)\ \hbox{ for all $t\in [0,t_n]$ and $x\in (-\kappa_n+n+t)P^{\e,\bar{x}_n}$}.$$

Finally, choose $\overline{R}>0$ such that, for all $n\in\N$,
$$\sum_{i=1}^{m}\eta M e^{-\lambda_i^n\overline{R}/2} \le\frac{\eta}{2},$$
which is possible because of~\eqref{lowerlambda} and $\lambda_i^n=\lambda_{f^n_i,b^n_i}$ with $c^*(f^n_i)\le b^n_i\le\varrho+1$ for all $n\in\N$ and $1\le i\le m$. For every $n\in\N$, since $B_{\e/2}(\bar{x}_n-\e\nu_n)\cup\{\bar{x}_n\}\subset P^{\e,\bar{x}_n}$ and $P^{\e,\bar{x}_n}$ is convex by Lemma~\ref{lemma:approximation}-(iii), one has that $B_{r\e/2}(\overline{x}_n- r\e\nu_n) \subset P^{\e,\bar{x}_n}$ for all $r\in [0,1]$. Then, one can take any $n$ large enough such that 
$$B_{\overline{R}/(2(\kappa_n-n-t_n))}\Big(\overline{x}_n-\frac{\overline{R}}{\kappa_n-n-t_n} \nu_n\Big)\subset P^{\e,\bar{x}_n},$$ 
whence
$$(x_n+\overline{R}\nu_n)\cdot f^n_i +(\kappa_n-n-t_n)b^n_i\ge \frac{\overline{R}}{2}\ \hbox{ for all $1\le i\le m$}.$$
Thus, $x_n+\overline{R}\nu_n\in (-\kappa_n+n+t_n)P_{\e,\bar{x}_n}$ for all $n$ large enough, and
$$\begin{array}{rcl}
v^n(0,\overline{R}\nu_n) & = & v_{\kappa_n}(t_n,x_n+\overline{R}\nu_n)\vspace{3pt}\\
& \le & \displaystyle u^+_n(t_n,x_n+\overline{R}\nu_n)\\
& \le & \displaystyle\sum_{i=1}^m \frac{\eta}{\displaystyle\min_{\R^N}\varphi^n_i} e^{-\lambda_i^n\overline{R}/2}\varphi_i^n(x_n+\overline{R}\nu_n)\le\sum_{i=1}^{m}\eta M e^{-\lambda_i^n\overline{R}/2}\le\frac\eta2.
\end{array}$$
This implies that $v^*(0,\overline{R}\nu)\le \eta/2$. One has finally reached a contradiction with $v^*(t,x)\equiv \eta$, see~\eqref{v*eta}. The proof of Lemma~\ref{lem:vk<e} is thereby complete.
\end{proof}

%%%%%%%%%%%%%%%%%%%%%%%%%%%%%%%%%%%%%%%%%%%%%%%%%%%

\subsection{The extinction region}\label{subsection4.2}

Throughout this section, we let $u$ be the solution of~\eqref{main} with the initial condition $u_0=\1_{U}$. Always assume that one of Hypotheses~\ref{hyp:KPP}-\ref{hyp:bistable} holds, and let $\W_0$ be given by~\eqref{w0}-\eqref{W0-wulff}. This section is the counterpart of Section~\ref{section3}. It deals with the proofs of Lemma~\ref{lem:upLS} and Proposition~\ref{prop:superset} below, which provide some upper estimates of $u$ in some suitable time-dependent sets. The proofs themselves rely on Lemmata~\ref{extract-kpp}-\ref{lem:vk<e}. 

\begin{lemma}\label{lem:upLS}
Assume that \eqref{supt-plane} holds. For any $\alpha>1$, one has
$$\sup_{x\in \R^N\setminus (U + t\alpha \W_0)} u(t,x)\rightarrow 0\ \hbox{ as $t\rightarrow +\infty$}.$$
\end{lemma}

\begin{proof}
Let $\alpha>1$ be fixed. Assume by way of contradiction that the conclusion does not hold. Then there exist $\eta>0$, a sequence $(t_n)_{n\in\N}$ of positive real numbers diverging to $+\infty$, and a sequence $(x_n)_{n\in\N}$ in $\R^N$ such that $x_n\in \R^N\setminus (U + t_n\alpha \W_0)$ and
\be\label{utnxneta}
u(t_n,x_n)> \eta>0
\ee
for all $n\in\N$. One then has 
\begin{equation}\label{xnW}
U-x_n\subset\R^N\setminus(-t_n\alpha\W_0).
\end{equation}
For each $n\in\N$, let $\bar{x}_n\in\Z^N$ be such that $x_n-\bar{x}_n\in[0,1]^N$.

For $\e>0$, let $K^\e:=\overline{\W_0}+\overline{B_\e}$ and $P^\e$ be the exterior $C^1$ and polyhedral approximations of $K:=\overline{\W_0}$ given in Lemma~\ref{lemma:approximation}. Remember that $\overline{\W_0}\subset P^\e\subset K^\e$. Since the compact set $\overline{W_0}$ is included into the open set $\alpha\W_0$, one can then fix $\e>0$ small enough so that
$$\overline{\W_0}\subset P^\e\subset K^\e\subset\alpha\W_0.$$
For these fixed values of $\e>0$ and $\eta>0$, let then $R>0$ be given by Lemma~\ref{extract-kpp} (in case of Hypothesis~\ref{hyp:KPP}) or by Lemma~\ref{lem:vk<e} (in case of Hypothesis~\ref{hyp:ignition} or~\ref{hyp:bistable}). Call also $C_\e:=P^\e$ in case of Hypothesis~\ref{hyp:KPP}, resp. $C_\e:=K^\e$ in case of Hypothesis~\ref{hyp:ignition} or~\ref{hyp:bistable}. Since $C_\e\supset\overline{\W_0}$ and $\W_0$ contains a ball centered at the origin, one can take $\beta\ge2$ large enough so that $(1-\beta)RC_\e\supset[0,1]^N$.

Since the compact set $C_\e$ is included into the open set $\alpha\W_0$, and since $t_n\to+\infty$ as $n\to+\infty$, there holds
\begin{equation}\label{tn2Re}
\frac{(t_n+\beta R)C_\e-[0,1]^N}{t_n}\subset \alpha \W_0\ \hbox{ for all large $n$}.
\end{equation}
For each $n\in\N$, call $\kappa_n:=t_n+\beta R$, whence $\kappa_n\ge t_n+2R\ge 2R$ and
$$(-\kappa_n+R+t_n)C_\e=(1-\beta)RC_{\e}\supset[0,1]^N.$$
It follows then from Lemma~\ref{extract-kpp} or~\ref{lem:vk<e} that, for each $n\in\N$,
\be\label{vkappan}
v_{-\kappa_nC_\e}(t_n,y)\le \eta\ \hbox{ for all }y\in[0,1]^N,
\ee
where we recall that the functions $v_{-\kappa_nC_\e}$ are the solutions of~\eqref{main} with initial conditions $\1_{\R^N\setminus(-\kappa_nC_\e)}$. On the other hand,~\eqref{tn2Re} implies that
$$-\big(\kappa_nC_\e-[0,1]^N\big)=-\big((t_n+\beta R)C_\e-[0,1]^N\big)\,\subset\,-t_n\alpha\W_0\ \hbox{ for all large $n$},$$
whence
$$\big(\R^N\setminus(-t_n\alpha\W_0)\big)+[0,1]^N\ \subset\ \R^N\setminus(-\kappa_nC_\e)\ \hbox{ for all large $n$}.$$
Therefore, there exists $n_1\in\N$ such that, for all $n\ge n_1$ and for all $x\in U$, one has $x-x_n\in\R^N\setminus(-t_n\alpha\W_0)$ by~\eqref{xnW}, and then $x-\bar{x}_n\in\R^N\setminus(-\kappa_nC_e)$ since $x_n-\bar{x}_n\in[0,1]^N$. As a consequence, $u_0=\1_U\le\1_{\R^N\setminus(-\kappa_nC_\e)}(\cdot-\bar{x}_n)=v_{-\kappa_nC_\e}(0,\cdot-\bar{x}_n)$ in $\R^N$ for all $n\ge n_1$. Since the coefficients of~\eqref{main} are $\Z^N$-periodic and $\bar{x}_n\in\Z^N$, it follows from the comparison principle that, for all $n\ge n_1$,
$$u(t,x)\le v_{-\kappa_nC_\e}(t,x-\bar{x}_n)\ \hbox{ for all $t\ge 0$ and $x\in\R^N$}.$$
Thus, for all $n\ge n_1$, there holds $u(t_n,x_n)\le v_{-\kappa_nC_\e}(t_n,x_n-\bar{x}_n)\le\eta$ by~\eqref{vkappan}. One has then reached a contradiction with~\eqref{utnxneta}, and the proof of Lemma~\ref{lem:upLS} is thereby complete.
\end{proof}

The previous lemma, combined with compactness and limit arguments, implies that the set $\mc{C}(\mathbb{S}^{N-1}\!\setminus\!\mc{B}(U))+\W_0$ is a spreading superset for $u$, in the sense given in the last sentence of Definition~\ref{def:W}.

\begin{proposition}\label{prop:superset}
Assume one of Hypotheses~\ref{hyp:KPP}-\ref{hyp:bistable}, and that~\eqref{supt-plane} holds. Then the set $\mc{C}(\mathbb{S}^{N-1}\!\setminus\!\mc{B}(U))+\W_0$ is a spreading superset for $u$.
\end{proposition}

\begin{proof}
Call $\W:=\mc{C}(\mathbb{S}^{N-1}\!\setminus\!\mc{B}(U))+\W_0$. If $\mc{B}(U)=\emptyset$, then $\W=\R^N$ and there is nothing to prove. One can then assume in the sequel that $\mc{B}(U)\neq\emptyset$, and then ${\rm int}(\R^N\!\setminus\W)\neq\emptyset$ since $\mc{B}(U)$ is relatively open in $\Sph$ and $\W_0$ is bounded. Take a non-empty compact set $C\subset {\rm int}(\R^N\!\setminus\!\W)$, any point $y\in C$, and then $\delta>0$ such that $\overline{B_{\delta}(y)}\subset {\rm int}(\R^N\!\setminus\!\W)$. We claim that 
\begin{equation}\label{Cy}
\liminf_{t\rightarrow +\infty} \frac{d\left(tB_{\delta}(y),U\!+\!t\W_0\right)}{t}=\liminf_{t\rightarrow +\infty} \frac{\inf\{|x\!-\!z|:x\in tB_{\delta}(y),\,z\in U\!+\!t\W_0\big\}}{t}>0.
\end{equation}
Then, assuming \eqref{Cy}, it follows that there is $\alpha>1$ such that $tB_{\delta}(y)\cap (U+t\alpha \W_0)=\emptyset$ for all $t>0$ large enough, whence
$$\lim_{t\rightarrow +\infty} \Big(\sup_{x\in B_{\delta}(y)} u(t,tx)\Big)=0$$
by Lemma~\ref{lem:upLS}. Since the compact set $C$ can be covered by a finite number of such open balls $B_{\delta}(y)$, one concludes that
$$\lim_{t\rightarrow +\infty} \Big(\max_{x\in C} u(t,tx)\Big)=0.$$

Thus, we only have to show \eqref{Cy}. Assume by way of contradiction that there are sequences $(\e_n)_{n\in\mathbb{N}}$ in $(0,+\infty)$ converging to $0^+$, $(t_n)_{n\in\mathbb{N}}$ in $(0,+\infty)$ diverging to~$+\infty$, $(y_n)_{n\in\mathbb{N}}$ in $B_{\delta}(y)$, $(z_n)_{n\in\mathbb{N}}$ in $U$ and $(\xi_n)_{n\in\mathbb{N}}$ in $\W_0$ such that $|t_ny_n -(z_n +t_n \xi_n)|<\e_n t_n$ for all $n\in\N$, that is,
$$\Big|y_n -\Big(\frac{z_n}{t_n} + \xi_n\Big)\Big|<\e_n.$$
There exist then $y_0\in \overline{B_{\delta}(y)}$ and $\xi_0\in \overline{\W_0}$ such that $y_n\rightarrow y_0$ and $\xi_n\rightarrow \xi_0$ as $n\to+\infty$, up to extraction of a subsequence. Since $\overline{B_{\delta}(y)}\subset{\rm int}(\R^N\!\setminus\!\W)=\R^N\!\setminus\!\overline{\W}\subset \R^N\!\setminus\!\overline{\W_0}$, one then gets that $y_0\neq \xi_0$. It also follows that
$$\frac{z_n}{t_n}\rightarrow y_0-\xi_0\ \hbox{ as }n\rightarrow +\infty,$$
which implies that
$$\lim_{n\rightarrow +\infty}\frac{d(t_n(y_0-\xi_0),U)}{t_n}= 0.$$
Owing to the definition of $\mathcal{B}(U)$, one has that $\hat{y_0-\xi_0}=(y_0-\xi_0)/|y_0-\xi_0|\in \mathbb{S}^{N-1}\setminus \mathcal{B}(U)$, whence $y_0\in\mc{C}(\mathbb{S}^{N-1}\!\setminus\! \mathcal{B}(U))+\overline{\W_0}\subset\overline{\W}$, a contradiction with $\overline{B_{\delta}(y)}\subset {\rm int}(\R^N\setminus \W)$. The proof of Proposition~\ref{prop:superset} is thereby complete.
\end{proof}

%%%%%%%%%%%%%%%%%%%%%%%%%%%%%%%%%%%%%%%%%%%%%%%%%%%
%%%%%%%%%%%%%%%%%%%%%%%%%%%%%%%%%%%%%%%%%%%%%%%%%%%

\section{Proofs of the main results}\label{section5}

We are now in a position to prove our main results, namely Theorems~\ref{Th:levelset},~\ref{Th:Ass},~\ref{Th:asspeed} and Corollary~\ref{theo:ass}.

%%%%%%%%%%%%%%%%%%%%%%%%%%%%%%%%%%%%%%%%%%%%%%%%%%%

\subsection{Proofs of Theorems~\ref{Th:levelset},~\ref{Th:Ass},~\ref{Th:asspeed} }\label{subsection5.1}

\begin{proof}[Proof of Theorem~\ref{Th:levelset}]
Fix $\lambda\in(0,1)$ and call for short
$$E_\lambda(t):=\big\{x\in\R^N\ :\ u(t,x)>\lambda\big\}.$$

On the one hand, by Lemma~\ref{lem:subset}, for any $0<\tau<1$, there exists $T>0$ such that
$u(t,x)>\lambda$ for all $t\geq T$ and all $x\in U_{\rho} +t\tau \W_0$, i.e.
$$\forall\,t\geq T,\quad U_{\rho} +t\tau \W_0\subset E_\lambda(t).$$
We deduce that
$$\forall\,t\geq T,\quad 
U_{\rho} +t\W_0\subset (U_{\rho} +t\tau \W_0)+t(1-\tau)\W_0
\subset E_\lambda(t)+ B_{t(1-\tau)\delta},$$
where $\delta>0$ is such that $\W_0\subset B_\delta$.

On the other hand, by Lemma~\ref{lem:upLS}, we deduce that for any $\alpha>1$, there exists $T'>0$ such that
$$\forall\,t\geq T',\quad E_\lambda(t) \subset U_{\rho} +t\alpha \W_0,$$
whence
$$\forall\,t\geq T',\quad 
E_\lambda(t) \subset 
U_{\rho} +t\alpha\W_0\subset (U_{\rho} +t\W_0)+t(\alpha-1)\W_0
\subset (U_{\rho} +t\W_0) + B_{t(\alpha-1)\delta},$$

This shows that
$$\forall\,t\geq \max\{T,T'\},\quad 
d_H(E_\lambda(t)\,,\, U_{\rho} +t\W_0)\leq t\delta\max\{1-\tau,\alpha-1\},$$
which concludes the proof by the arbitrariness of $\tau<1$ and $\alpha>1$.
\end{proof}

As for Theorem~\ref{Th:Ass}, it follows immediately from Propositions~\ref{prop:subset} and~\ref{prop:superset}.

\begin{proof}[Proof of Theorem~\ref{Th:Ass}]
Since $\mc{B}(U)\cup \mc{U}(U_{\rho})=\mathbb{S}^{N-1}$ by assumption~\eqref{BUS}, and since $\mathcal{U}(U_{\rho})\subset\mathcal{U}(U)$ and $\mathcal{U}(U)$ and $\mathcal{B}(U)$ are disjoint, one has that $\Sph\!\setminus\!\mc{B}(U)=\mc{U}(U_{\rho})=\mc{U}(U)$. Owing to Definition~\ref{def:W}, the conclusion then follows directly from Propositions~\ref{prop:subset} and~\ref{prop:superset}.
\end{proof}

Next, Theorem~\ref{Th:asspeed} is a formulation of Theorem~\ref{Th:Ass}.

\begin{proof}[Proof of Theorem~\ref{Th:asspeed}]
We only have to show that the spreading set $\W=\mc{C}(\mathcal{U}(U))+\W_0$ satisfies
$$\W=\big\{r\xi:\xi\in\mathbb{S}^{N-1},\, 0\le r<\omega(\xi)\big\},$$
with $\omega(\xi)$ given by \eqref{ASS-G}. To do so, we show this equality ray by ray.

For $\xi \in \mathcal{U}(U)$, since $0\in\W_0$, one obviously has that $r\xi\in \W$ for any $0\le r<+\infty=\omega(\xi)$.

Assume in the sequel that $\xi\in\Sph\!\setminus\!\mc{U}(U)$, that is, $\xi\in\mc{B}(U)$ thanks to~\eqref{BUS}. Take any $r\in[0,\omega(\xi))$. From the definition of $\omega(\xi)$, there is then $z\in\mc{C}(\mc{U}(U))$ such that
$$r<\frac{\omega_0(\hat{\xi-z})}{|\xi-z|}.$$
In particular, $|\xi-z|>0$ and $r\xi=rz+r(\xi-z)$ with $rz\in\mc{C}(\mc{U}(U))$ and $r\,|\xi-z|<\omega_0(\hat{\xi-z})$. This implies that $r\xi\in\mc{C}(\mathcal{U}(U))+\W_0=\W$.

Conversely, take any $r\ge\omega(\xi)$, whence $r\ge\omega_0(\xi)>0$. For any $z\in\mc{C}(\mc{U}(U))$, one has $z/r\in\mc{C}(\mc{U}(U))$, $\xi-z/r\neq0$ (since $\xi\in\Sph\!\setminus\!\mc{U}(U)$), and
$$|r\xi-z|=r\,\big|\xi-z/r\big|\ge\omega_0\big(\hat{\xi-z/r}\big)=\omega_0\big(\hat{r\xi-z}\big).$$
Hence, by~\eqref{W0}, $r\xi-z\not\in\W_0$ for every $z\in\mc{C}(\mc{U}(U))$, that is, $r\xi\not\in \mc{C}(\mc{U}(U))+\W_0=\W$. The proof of Theorem~\ref{Th:asspeed} is thereby complete.
\end{proof}

%%%%%%%%%%%%%%%%%%%%%%%%%%%%%%%%%%%%%%%%%%%%%%%%%%%

\subsection{$V$ and $\Lambda$-shaped initial data: proof of Corollary~\ref{theo:ass}}\label{subsection5.2}

Corollary~\ref{theo:ass} is a consequence of Theorem~\ref{Th:Ass} for V-shaped or $\Lambda$-shaped initial supports~$U$. We divide its proof into these two different cases.

\begin{proof}[Proof of Corollary~\ref{theo:ass} when $U$ is V-shaped]
We first assume that~\eqref{BUS} holds and that $U$ is V-shaped. Hence, $\mc{B}(U)=\Sph\setminus\mc{U}(U)$, $\emptyset\neq\mc{B}(U)\subsetneqq\Sph$, $\mc{C}(\mc{B}(U))$ is convex, and $\emptyset\neq\mc{L}(\mc{B}(U))\subsetneqq\Sph$. But we do not assume~\eqref{supt-plane} for the moment (this assumption will only be used at the end of the proof). 

Let us first prove in this paragraph that
\be\label{EqF-V}
\underbrace{\Big\{x\in\R^N:\, \inf_{e\in\mathcal{L}(\mathcal{B}(U))}(x\cdot e -c^*(e))< 0\Big\}}_{=:\W_V}=\underbrace{\big\{r\xi: \xi\in\mathbb{S}^{N-1},\, 0\le r<\omega_V(\xi)\big\}}_{=:\W_1},
\ee
with $\omega_V(\xi)$ given by \eqref{ASS-V}. We show the equality between these two sets ray by ray. Firstly, for any $\xi\in\mathcal{U}(U)$ and any $r\ge0$, there holds that $\inf_{e\in\mathcal{L}(\mathcal{B}(U))} \{r\xi\cdot e\}\le 0$ (otherwise $r\xi\in{\rm{int}}(\mc{C}(\overline{\mc{B}(U)}))={\rm{int}}(\overline{\mc{C}(\mc{B}(U))})$ by Lemma~\ref{lemma:P}, and then $r>0$ and $r\xi\in\mc{C}(\mc{B}(U))$ by convexity of $\mc{C}(\mc{B}(U))$, and finally $\xi\in\mc{B}(U)$, a contradiction). It then follows from the compactness of $\mc{L}(\mc{B}(U))$ and the positivity of the speeds $c^*(e)$'s that 
$$\inf_{e\in\mathcal{L}(\mathcal{B}(U))} \{r\xi\cdot e -c^*(e)\}<0.$$
Thus,
$$r\xi\in\W_V\cap\W_1\ \hbox{ for all $\xi\in\mc{U}(U)$ and $0\le r<\omega_V(\xi)=+\infty$}.$$
Secondly, remember that, from Lemma~\ref{lemma:P} and the relative openness of $\mc{B}(U)$ and closedness of $\mc{L}(\mc{B}(U))$ in $\Sph$, one has $\inf_{e\in\mc{L}(\mc{B}(U))}\{e\cdot\xi\}>0$ for all $\xi\in\mc{B}(U)$. Therefore, for any $\xi\in\mc{B}(U)$ and any $r\in [0,\omega_V(\xi))$, there is $e\in\mathcal{L}(\mathcal{B}(U))$ (hence, $e\cdot \xi>0$) such that
$$r<\frac{c^*(e)}{e\cdot \xi},$$
that is, $r\xi\cdot e -c^*(e)<0$ and then $r\xi\in\W_V$. On the other hand, for any $\xi\in \mathcal{B}(U)$ and $r\in [\omega_V(\xi),+\infty)$, it follows that $e\cdot\xi>0$ and $r\ge c^*(e)/(e\cdot \xi)$ for all $e\in\mathcal{L}(\mathcal{B}(U))$, whence $\inf_{e\in\mathcal{L}(\mathcal{B}(U))}\{r\xi\cdot e -c^*(e)\}\ge 0$ and $r\xi\not\in\W_V$. This completes the proof of~\eqref{EqF-V}.

Next, we show in this paragraph that
\be\label{hatW}
\mc{C}(\mathcal{U}(U))+\W_0=\underbrace{\Big\{x\in\R^N:\, \inf_{e\in\mathcal{L}(\mathcal{B}(U))}(x\cdot e-\hat{c}(e))< 0\Big\}}_{=:\W_2},
\ee
where
\be\label{defhatc}
\hat{c}(e):=\sup_{\xi\in\mathbb{S}^{N-1}} \omega_0(\xi)\xi\cdot e.
\ee
Take first any $x\in\W_2$, and set $\e:=-\inf_{e\in\mathcal{L}(\mathcal{B}(U))}(x\cdot e-\hat{c}(e))>0$. Then, choosing $e\in\mathcal{L}(\mathcal{B}(U))$ such that $x\cdot e-\hat{c}(e)<-\e/2$, it follows from~\eqref{defhatc} that there are $\xi\in\mathbb{S}^{N-1}$ and $\delta\in(0,\omega_0(\xi))$ small enough such that $(\omega_0(\xi)-\delta)\xi\cdot e>\hat{c}(e)-\e/2$. Thus,
$$(x-(\omega_0(\xi)-\delta)\xi)\cdot e=x\cdot e -(\omega_0(\xi)-\delta)\xi\cdot e<x\cdot e -\hat{c}(e)+\frac{\e}{2}<0.$$
Lemma~\ref{lemma:P} implies that $y:=x-(\omega_0(\xi)-\delta)\xi\not\in\mc{C}(\overline{\mc{B}(U)})$, whence $y\in\mc{C}(\mathcal{U}(U))$ and $x=y+(\omega_0(\xi)-\delta)\xi\in\mc{C}(\mathcal{U}(U))+\W_0$. Therefore,
\be\label{subset2}
\W_2\subset\mc{C}(\mathcal{U}(U))+\W_0.
\ee
Take now any $x\in\R^N\setminus\W_2$, that is,
$$\inf_{e\in\mathcal{L}(\mathcal{B}(U))}(x\cdot e-\hat{c}(e))\ge 0,$$
and take any $y\in\mc{C}(\mathcal{U}(U))$. One has in particular $\inf_{e\in\mathcal{L}(\mathcal{B}(U))}x\cdot e>0$ (since $\hat{c}(e)\ge\omega_0(e)$ and the $\omega_0(e)$'s are bounded from below by a positive constant), whence $x\neq0$ and $x\in{\rm{int}}(\mc{C}(\overline{\mc{B}(U)}))={\rm{int}}(\overline{\mc{C}(\mc{B}(U))})$ by Lemma~\ref{lemma:P}, and then $x\in\mc{C}(\mc{B}(U))\setminus\{0\}$ by convexity of $\mc{C}(\mc{B}(U))$. As a consequence, $x\neq y$. Call then $\xi:=\hat{x-y}=(x-y)/|x-y|$. One has $x=y +|x-y|\xi$ and, for every $e\in\mathcal{L}(\mathcal{B}(U))$,
\be\label{out-hatW}
0\le x\cdot e-\hat{c}(e)=y\cdot e +|x-y|\xi\cdot e-\hat{c}(e).
\ee
Since $y\in\mc{C}(\mathcal{U}(U))=(\R^N\setminus\mc{C}(\mathcal{B}(U)))\cup\{0\}$, Lemma~\ref{lemma:P} yields the existence of $e\in\mathcal{L}(\mathcal{B}(U))$ such that $y\cdot e\le0$. Since $\hat{c}(e)>0$, it follows from~\eqref{out-hatW} that $\xi\cdot e >0$, whence
$$|x-y|\ge\frac{\hat{c}(e)}{\xi\cdot e}\ge \omega_0(\xi)=\omega_0(\hat{x-y}).$$
Therefore, $x-y\not\in\W_0$, and thus $x\not\in\mc{C}(\mathcal{U}(U))+\W_0$. One has then shown $\R^N\setminus\W_2\subset\R^N\setminus(\mc{C}(\mathcal{U}(U))+\W_0)$. Together with~\eqref{subset2}, this completes the proof of~\eqref{hatW}.

Furthermore, as in the proof of~\eqref{EqF-V}, it follows from~\eqref{hatW} that
\be\label{hatW2}
\mc{C}(\mathcal{U}(U))+\W_0=\big\{r\xi: \xi\in\mathbb{S}^{N-1},\, 0\le r<\hat{\omega}(\xi)\big\},
\ee
where~$\hat{\omega}(\xi)$ is defined as in~\eqref{hatomega}.

Let us finally complete the proof of Corollary~\ref{theo:ass}-(i), by assuming now that~\eqref{supt-plane} holds as well. On the one hand,~\eqref{w0} implies that $\hat{c}(e)=\sup_{\xi\in\mathbb{S}^{N-1}} \omega_0(\xi)\xi\cdot e\le c^*(e)$ for every $e\in\Sph$. On the other hand, condition \eqref{supt-plane} implies that, for every $e\in\mathbb{S}^{N-1}$, there is $\xi\in\mathbb{S}^{N-1}$ such that $\omega_0(\xi)\xi\cdot e=c^*(e)$, whence $\hat{c}(e)\ge c^*(e)$. Thus, $\hat{c}(e)=c^*(e)$ for all $e\in\Sph$, and it follows from~\eqref{EqF-V} and~\eqref{hatW} that $\W_V=\mc{C}(\mathcal{U}(U))+\W_0$. By Theorem~\ref{Th:Ass}, the set $\W_V$ is then the spreading set, and the proof of Corollary~\ref{theo:ass} when $U$ is V-shaped is thereby complete.
\end{proof}

\begin{proof}[Proof of Corollary~\ref{theo:ass} when $U$ is $\Lambda$-shaped]
We first assume that~\eqref{BUS} holds and that~$U$ is $\Lambda$-shaped. Hence, $\mc{B}(U)=\Sph\setminus\mc{U}(U)$, $\emptyset\neq\mc{U}(U)\subsetneqq\Sph$, $\mc{C}(\mc{U}(U))$ is convex, and $\emptyset\neq\mc{L}(\mc{U}(U))\subsetneqq\Sph$. But we do not assume~\eqref{supt-plane} for the moment (this assumption will only be used at the end of the proof). 

We first prove in this paragraph that
\be\label{EqF-L}
\underbrace{\Big\{x\in\R^N:\, \displaystyle\inf_{e\in\mathcal{L}(\mathcal{U}(U))}(x\cdot e+ c^*(-e))> 0\Big\}}_{=:\W_\Lambda}=\underbrace{\big\{r\xi:\xi\in\mathbb{S}^{N-1},\, 0\le r<\omega_{\Lambda}(\xi)\big\}}_{=:\W_3},
\ee
where $\omega_{\Lambda}(\xi)$ is given by \eqref{omega-Lambda}. We show the equality between these two sets ray by ray. Firstly, for any $\xi\in\mathcal{U}(U)$ and any $r\ge0$, Lemma~\ref{lemma:P} implies that $\inf_{e\in\mathcal{L}(\mathcal{U}(U))}\{r\xi\cdot e\}\ge 0$, and it follows from the uniform positivity of the speeds $c^*(-e)$'s that 
$$\inf_{e\in\mathcal{L}(\mathcal{U}(U))} \{r\xi\cdot e +c^*(-e)\}>0.$$
Thus,
$$r\xi\in\W_\Lambda\cap\W_3\ \hbox{ for all $\xi\in\mc{U}(U)$ and $0\le r<\omega_\Lambda(\xi)=+\infty$}.$$
Take now any $\xi\in\mathcal{B}(U)$ and any $r\in [0,\omega_{\Lambda}(\xi))$. Then $\xi\not\in\mc{C}(\mc{U}(U))=\mc{C}(\overline{\mc{U}(U)})$ and Lemma~\ref{lemma:P} yields the existence of $e\in\mathcal{L}(\mathcal{U}(U))$ such that $e\cdot \xi<0$. For any such $e$, one then has by~\eqref{omega-Lambda} that
$$r<\frac{c^*(-e)}{-e\cdot\xi},$$
that is, $r\xi\cdot e+c^*(-e)>0$. Moreover, for any $e'\in\mathcal{L}(\mathcal{U}(U))$ with $e'\cdot\xi\ge0$, one also has $r\xi\cdot e'\!+\!c^*(-e')\!\ge\!c^*(-e')\!>\!0$. From the compactness of $\mathcal{L}(\mathcal{U}(U))$ and the continuity of the map $e\mapsto c^*(-e)$ in $\Sph$, one gets that $r\xi\in\W_\Lambda$. On the other hand, for any $\xi\in\mc{B}(U)$, any $r\in [\omega_{\Lambda}(\xi),+\infty)$, and any $\e>0$, there is $e\in \mathcal{L}(\mathcal{U}(U))$ such that~$e\cdot \xi<0$~and
$$r\ge\frac{c^*(-e)}{-e\cdot \xi}-\e,$$
whence $r\xi\cdot e +c^*(-e)\le-\e(e\cdot\xi)\le\e$. Since $\e>0$ can be arbitrarily small, this means that $\inf_{e\in\mathcal{L}(\mathcal{U}(U))}\{r\xi\cdot e +c^*(-e)\}\le 0$ for any $\xi\in\mc{B}(U)$ and any $r\in [\omega_{\Lambda}(\xi),+\infty)$, whence~$r\xi\not\in\W_\Lambda$. This completes the proof of~\eqref{EqF-L}.

Next, we show in this paragraph that
\be\label{hatW-L}
\mc{C}(\mathcal{U}(U))+\W_0=\underbrace{\Big\{x\in\R^N:\, \inf_{e\in\mathcal{L}(\mathcal{U}(U))}(x\cdot e+\hat{c}(-e))> 0\Big\}}_{=:\W_4},
\ee
where we recall from~\eqref{defhatc} that $\hat{c}(e)=\sup_{\xi\in\mathbb{S}^{N-1}} \omega_0(\xi)\xi\cdot e\ge\omega_0(e)>0$. Take first any $x\in\R^N\setminus\W_4$, that is,
\be\label{infhatc}
\inf_{e\in\mathcal{L}(\mathcal{U}(U))}(x\cdot e+\hat{c}(-e))\le 0,
\ee
and take any $y\in\mc{C}(\mathcal{U}(U))$. Lemma~\ref{lemma:P} implies that $y\cdot e\ge0$ for all $e\in\mathcal{L}(\mathcal{U}(U))$ and, together with $\hat{c}(-e)\ge\omega_0(-e)$ and the uniform positivity of the $\omega_0(-e)$'s, one gets that $\inf_{e\in\mathcal{L}(\mathcal{U}(U))}\{y\cdot e+\hat{c}(-e)\}>0$. Therefore, $x\neq y$. Call $\xi:=\hat{x-y}=(x-y)/|x-y|$ and consider any $\eta\in(0,1)$. Then, $x=y +|x-y|\xi$ and, since
$$\inf_{e\in\mathcal{L}(\mathcal{U}(U))}(x\cdot e+\hat{c}(-e)-\eta\,\hat{c}(-e))<0$$
from~\eqref{infhatc} and the uniform positivity of the $\hat{c}(-e)$'s, there is $e\in\mathcal{L}(\mathcal{U}(U))$ such that
$$\eta\,\hat{c}(-e)\ge x\cdot e+\hat{c}(-e)=y\cdot e +|x-y|\xi\cdot e+\hat{c}(-e),$$
whence 
$$0\ge y\cdot e +|x-y|\xi\cdot e+(1-\eta)\,\hat{c}(-e).$$
Since $y\in\mc{C}(\mathcal{U}(U))$ and $(1-\eta)\,\hat{c}(-e)>0$, one has $y\cdot e\ge 0$ and then $\xi\cdot e <0$, whence
$$|x-y|\ge\frac{-(1-\eta)\,\hat{c}(-e)}{\xi\cdot e}\ge(1-\eta)\,\omega_0(\xi).$$
Since this is true for any $\eta\in(0,1)$, one infers that $|x-y|\ge\omega_0(\xi)$, i.e. $x-y\not\in\W_0$. Therefore, $x\not\in\mc{C}(\mathcal{U}(U))+\W_0$. One has then shown that $\R^N\setminus\W_4\subset\R^N\setminus(\mc{C}(\mathcal{U}(U))+\W_0)$, that is,
\be\label{subset1}
\mc{C}(\mathcal{U}(U))+\W_0\subset\W_4.
\ee
Take now any $x\in\R^N\setminus\overline{\mc{C}(\mathcal{U}(U))+\W_0}$, let $y$ be the projection of $x$ onto the closed convex set $\overline{\mc{C}(\mathcal{U}(U))+\W_0}$, and call
$$e:=\hat{y-x}=\frac{y-x}{|y-x|}.$$
One has that $e\cdot z\ge e\cdot y$ for all $z\in\overline{\mc{C}(\mathcal{U}(U))+\W_0}$. In particular, for any $\xi\in\mc{U}(U)$, since $0\in\W_0$, one has $e\cdot(t\xi)\ge e\cdot y$ for all $t\ge0$, whence $e\cdot\xi\ge0$. Therefore, $e\in\mc{L}(\mc{U}(U))$. Since $0\in\mc{C}(\mathcal{U}(U))$, one also gets that $\overline{\W_0}\subset\overline{\mc{C}(\mathcal{U}(U))+\W_0}$ and then $e\cdot(\omega_0(\xi)\xi)\ge e\cdot y$ for all $\xi\in\Sph$, whence $\hat{c}(-e)+e\cdot y\le0$. As a consequence,
$$\hat{c}(-e)+e\cdot x\le e\cdot (x-y)=-|x-y|<0$$
and $x\not\in\W_4$. One has then shown that $\R^N\setminus\overline{\mc{C}(\mathcal{U}(U))+\W_0}\subset\R^N\setminus\W_4$, that is, $\W_4\subset\overline{\mc{C}(\mathcal{U}(U))+\W_0}$. Since $\W_4$ is open and $\mc{C}(\mathcal{U}(U))+\W_0$ is open (because $\W_0$ is open) and convex (as the sum of two convex sets), one gets that $\W_4\subset{\rm{int}}(\overline{\mc{C}(\mathcal{U}(U))+\W_0})=\mc{C}(\mathcal{U}(U))+\W_0$. Together with~\eqref{subset1}, this completes the proof of~\eqref{hatW-L}.

Finally, assuming condition~\eqref{supt-plane}, one has $\hat{c}(e)=c^*(e)$ for all $e\in\mathbb{S}^{N-1}$, and it follows from \eqref{EqF-L} and \eqref{hatW-L} that $\W_{\Lambda}=\mc{C}(\mathcal{U}(U))+\W_0$. By Theorem~\ref{Th:Ass}, the set $\W_{\Lambda}$ is then the spreading set. The proof of Corollary~\ref{theo:ass} is thereby complete.  
\end{proof}

%-----------------------------------------------------------------------------------
%-----------------------------------------------------------------------------------

%-------------------------------------------------------------------------------
%-------------------------------------------------------------------------------

\end{document}